\documentclass[12pt]{article}

\usepackage{amssymb}
\usepackage{amsthm}
\usepackage{amsmath}
\usepackage[retainorgcmds]{IEEEtrantools}
\usepackage[margin=2.5cm]{geometry}
\usepackage{makeidx}
\usepackage{amsfonts}
\usepackage{enumerate}

\usepackage{wrapfig}
\usepackage{here}
\usepackage{graphicx}
\usepackage{tikz}
\usepackage%[dvips,colorlinks,citecolor=red,linkcolor=red]
{hyperref}

\theoremstyle{plain}
\newtheorem{theorem}{Theorem}[section]
\newtheorem{claim}{Claim}

\newtheorem{lemma}[theorem]{Lemma}
\newtheorem{corollary}[theorem]{Corollary}

\newtheorem{proposition}[theorem]{Proposition}

\newtheorem{problem}{Problem}
\theoremstyle{definition} 

\definecolor{blau}{rgb}{0.1,0.0,0.9}
\definecolor{gruen}{cmyk}{1.0,0.2,0.7,0.07}
\definecolor{mag}{cmyk}{0.0,0.9,0.3,0.0}

\numberwithin{equation}{section}

\begin{document}

\date{\today}
\title{Solution of Vizing's Problem on Interchanges
for Graphs with Maximum Degree 4 and Related Results}
%with maximum degree $4$}
\author{{\sl Armen S. Asratian}\thanks{{\it E-mail address:} armen.asratian@liu.se}\\ 
Department of Mathematics \\
Link\"oping University \\ 
SE-581 83 Link\"oping, Sweden
\and
{\sl Carl Johan Casselgren}\thanks{{\it E-mail address:} 
carl.johan.casselgren@liu.se \, Part of the work done
while the author was a postdoc at Syddansk Universitet.
Research supported by SVeFUM.} \\
Department of Mathematics \\
Link\"oping University \\
SE-581 83 Link\"oping, Sweden
}

\maketitle

\bigskip 
\noindent
{\bf Abstract.}  
	Let $G$ be a Class 1 graph with maximum degree $4$
	and let $t\geq 5$  be an integer.
	We show that  any proper $t$-edge coloring of $G$  can be transformed to
	any proper $4$-edge coloring of $G$ using only transformations
	on $2$-colored subgraphs (so-called interchanges).
	This settles the smallest previously 	unsolved case of
	a well-known  problem of Vizing
	on interchanges, posed in 1965. Using our result we give an
	affirmative answer to a question of Mohar
	for two classes of graphs: we show  that
	all proper $5$-edge colorings of a Class 1 graph with
	maximum degree 4 are Kempe equivalent,
	that is, can be transformed to each other by interchanges, and that
 	all proper 7-edge colorings of a
	Class 2 graph with maximum degree 5 are Kempe equivalent. 	\bigskip

\section{Introduction}

  We consider finite graphs without loops and multiple edges.
	A proper $t$-edge coloring  of a graph $G$ is a
	mapping $f:E(G)\longrightarrow  \{1,\dots,t\}$ such that
	$f(e)\not=f(e')$ for every  pair of adjacent edges $e$ and $e'$ in $G$.
	If $e\in E(G)$ and $f(e)=k$ then  
	we say
	that the edge $e$ is {\em colored $k$ under $f$}. 
	The set of edges colored $k$ under $f$
	we denote by $M(f,k)$. 	We denote by $\Delta(G)$ the maximum degree of
	vertices of a graph $G$, and by $d_G(v)$ the degree of a 
	vertex $v$ of $G$. 
	The {\em chromatic index $\chi'(G)$} of a graph $G$
	is the minimum number  $k$ for which there exists  a proper
	$k$-edge coloring of $G$.
	 For a proper $t$-edge coloring $f$ of $G$ and  any two colors
	$a,b \in \{1,\dots,t\}$ we denote by $G_f(a,b)$ the subgraph induced by the set 
	$M(f,a)\cup M(f,b)$.  By switching colors
	$a$ and $b$ on a connected component of $G_f(a,b)$,
	we obtain another proper $t$-edge coloring of $G$.
	This operation is called an {\em interchange} or  a {\em Kempe change}.
	 Two edge colorings $f,g$
	are {\em  Kempe equivalent} if $g$ can be obtained
	from $f$ by a sequence of interchanges.
	
	Interchanges play a key role in investigations on edge colorings.
	The proofs of many results in this area are based on transformations of one
	proper edge coloring of a graph $G$ to another using interchanges.
	 For example, Vizing's  theorem on the 
	chromatic index \cite{Vizing1} can be reformulated,
	taking into consideration the proof of  it,  in the following way.\medskip

\noindent {\bf Theorem A}.  For every  graph $G$,
$\chi'(G)\le \Delta(G)+1$. Moreover every proper $t$--edge coloring of $G$,  
where $t\geq
\Delta(G)+2$, can be transformed to a proper $(\Delta(G)+1)$--edge coloring  
of $G$ by
a sequence of interchanges.\bigskip

Vizing's result implies that for any graph $G$, $\chi'(G) = \Delta(G)$
	or $\chi'(G) = \Delta(G) +1$. In the former case
	$G$ is said to be \emph{Class $1$}, and
	in the latter $G$ is {\em Class $2$}.
For a  Class 2 graph $G$ 
 the result of Vizing means that  any
proper edge coloring of $G$ can be transformed to a proper  
$\chi'(G)$-edge coloring
by using interchanges only. 
The following problem posed by Vizing \cite{Vizing2, Vizing3}
	(see also \cite{StiebitzToft, JensenToft}) for Class 1 graphs is still open:

\begin{problem}
\label{prob:vizing}
	Is it true that any Class 1 graph $G$ satisfies the following property: every proper $t$-edge coloring of $G$,  
$t\geq
\Delta(G)+1$, can be transformed to a proper $\Delta(G)$-edge coloring  
of $G$ by
a sequence of interchanges? (This property we shall call the Vizing property.)
\end{problem}

	Mohar \cite{Mohar1} used the edge coloring algorithm suggested by
	Vizing \cite{Vizing2} to prove that
	for a  graph $G$
	with a  proper $t$-edge coloring
	$f$, every proper $\chi'(G)$-edge coloring can
	be obtained from $f$ by a sequence of interchanges, provided that $t \geq \chi'(G)+2$.
	This means that all proper $t$-edge colorings of $G$ are Kempe equivalent if $t\geq
\chi'(G)+2$. 
	Mohar also posed the following:

\begin{problem}
\label{prob:mohar}
 Is it true that all proper $(\chi'(G)+1)$-edge colorings 
 of a  graph $G$ are Kempe equivalent?
\end{problem}

	For  graphs $G$ with $\Delta(G) \leq 3$, 
	Problem \ref{prob:mohar}   
	was resolved to the positive by McDonald et al. 
	\cite{McDonald}.  In fact they  proved  a stronger result:
\medskip
	
	\noindent {\bf Theorem B} \cite{McDonald}. Let $G$ be a  
	graph with $\Delta(G)\leq 3$. 
	Then all proper $(\Delta(G)+1)$-edge colorings of  $G$ are Kempe equivalent.\medskip

	In the general case the situation is different. It is shown 
	in \cite{McDonald} 		that if $\chi'(G)$ is replaced by $\Delta(G)$
	in Problem \ref{prob:mohar}, then we get a problem which in
	general has a negative answer. 

	McDonald et al. \cite{McDonald} obtained also 
 an affirmative answer to Problem 2 for some Class 2 graphs:

\noindent {\bf Theorem C} \cite{McDonald}. Let $G$ be a  graph 
	with $\Delta(G)=4$
	and $\chi'(G)=5$. 
	Then all proper $6$-edge colorings of  
	$G$ are Kempe equivalent.\medskip

  Our first result establishes a connection between Problem 1 and Problem 2.
  
 % \label{th:connect}
 %	Let $k\geq 2$ be an integer and let  every Class 1 graph with maximum degree not exceeding $k$
%	have the Vizing property. Then 
% \begin{itemize}
%	\item[(a)] 
% for every graph $G$ with $\Delta(G)\leq k$
 %	all proper $(\chi'(G)+1)$-edge colorings are Kempe equivalent.
%\item[(b)]	for every Class 2  graph $H$ with $\Delta(H)=k+1$ all proper $(k+3)$-edge colorings of $H$ 
%	are Kempe equivalent.
	
%		\end{itemize}
% \end{theorem}

  \begin{theorem}
  \label{th:connect}
  Let $k\geq 2$ be an integer. Then all Class 1 graphs with 
  maximum degree at most $k$ 
	have the Vizing property if and only if 
  for every graph $G$ with $\Delta(G)\leq k$
 	all proper $(\chi'(G)+1)$-edge colorings are Kempe equivalent.
 \end{theorem}

 Our second result establishes a connection between the solution of  Problem 2
 for graphs with maximum degree $k-1$ 
 and Class 2 graphs with maximum degree $k$.

\begin{theorem}
  \label{th:connect1}
   Let $k\geq 4$ be an integer. If for every graph $H$ with $\Delta(H)=k-1$
 	all proper $(\chi'(H)+1)$-edge colorings are Kempe equivalent,
	then for every Class 2 graph  $G$ with maximum degree $k$ all proper $(k+2)$-edge colorings are Kempe equivalent.
 \end{theorem}

	In the case $k=4$, Theorem \ref{th:connect1} actually  
	implies that Theorem C follows from Theorem B.

	The main result of  this paper is  the following  theorem which 
	gives a positive answer 
	to the smallest previously unsolved case
	of Problem 1.

\begin{theorem}
\label{th:main}	
	Let $G$ be a  Class 1 graph with $\Delta(G)=4$ and 
	$t$ be an integer, $t\geq 5$. Then every proper
	$t$-edge coloring of $G$ 
	can be transformed to any proper  $4$-edge coloring  of $G$  
	by a sequence of 	interchanges. 
\end{theorem}

	By Theorem \ref{th:main}, any two proper $5$-edge colorings of a Class 1
	graph $G$ with $\Delta(G) =4$ can be transformed by interchanges
	to the same proper $4$-edge coloring of $G$. This implies
	the next result, which gives a positive answer to the smallest 
	previously unsolved
	case of Problem 2.
	%The next result follows from Theorem \ref{th:main} and
	%gives a positive answer to the smallest unsolved case of Problem 2.

\begin{corollary}
\label{cor:equiv}
 If $G$ is a  Class 1 graph with $\Delta(G)=4$, then all  proper $5$-edge colorings of
 $G$ are Kempe equivalent.
\end{corollary}

%\begin{proof}
%	Let $f$ and $g$ be proper $5$-edge colorings of $G$, and $\varphi$
%	a proper $4$-edge coloring of $G$. Then $f$ and $\varphi$ are
%	Kempe equivalent by Theorem \ref{th:main}, and the same holds for
%	$g$ and $\varphi$. Hence $f$ and $g$ are Kempe equivalent as well.
%\end{proof}

Theorem \ref{th:connect1}, Theorem C and Corollary \ref{cor:equiv} imply the following result:\medskip

\begin{corollary}
\label{th:reduction}
 	Let $G$ be a Class 2 graph with $\Delta(G)=5$. Then all proper 7-edge colorings 
	of $G$ are Kempe equivalent.
\end{corollary}

	 Note that Corollary \ref{cor:equiv}
	and Theorem B  imply that
	Problem 2 has an affirmative answer for all graphs
	with maximum degree at most $4$. Additionally,
	Corollary \ref{th:reduction} settles the question
	for Class 2 graphs with maximum degree $5$,
	so 
	the smallest unsolved
	case is Class 1 graphs of maximum degree $5$.

	Next, we describe some Class 1 graphs $G$ with $\Delta(G)\geq 5$
	for which all proper $(\chi'(G)+1)$-edge colorings are Kempe equivalent.
	Denote by $G_{\geq 5}$ the subgraph of $G$ induced by the set of vertices with
		degree at least $5$.
		
	\begin{theorem}
\label{th:main3}	
	For every  graph $G$ with $\Delta(G)\geq 5$ where the subgraph 
	$G_{\geq 5}$ is acyclic,
	all proper $(\chi'(G)+1)$-edge colorings of $G$ are Kempe equivalent.
	\end{theorem}

		By Corollary \ref{cor:equiv}, 
		all proper $5$-edge colorings of a Class 1
	 graph  with maximum degree $4$ are Kempe equivalent. 
	Let us now briefly consider the problem of transforming a
	proper $4$-edge coloring of such a graph $G$
	to another proper $4$-edge coloring of $G$.
	The next example shows that there are such graphs 
	with two proper $4$-edge
	colorings $f$ and $g$ such that $f$ cannot be transformed to $g$  
	using not only 	interchanges but also transformations on
	3-edge colored subgraphs.

\begin{figure} %[H]
\label{fig:graph}
\begin{center}
\begin{tikzpicture}[thick,scale=1.1]

\coordinate(x_1) at (0,2);
%\node at (-0.5,2) {$x_1$};
\node at (-0.7,1.7) {$1$};
\node at (0.75,2.3) {$2$};
\node at (0.1,0.4) {$3$};
\coordinate(x_2) at (1.5,2);
%\node at (2,2) {$x_2$};
\node at (2.2,1.7) {$1$};
\node at (1.37,1.35) {$3$};
\node at (1.0,1.65) {$4$};
\node at (0.7,1.35) {$4$};
\coordinate(x_3) at (2.5,1);
\node at (2.2,0.4) {$2$};
%\node at (3,1) {$x_3$};
\coordinate(x_4) at (1.5,0);
\node at (0.6,0.45) {$4$};
\node at (0.75,-0.3) {$1$};
%\node at (1.5,-0.5) {$x_4$};
\coordinate(x_5) at (0,0);
\node at (-0.8,0.4) {$2$};
%\node at (0,-0.5) {$x_5$};
\coordinate(x_6) at (-1,1);
%\node at (-1.5,1) {$x_6$};
\node at (-0.4,0.78) {$3$};

%\node at (-1.5,-1.5) {The graph $G$ with an edge coloring $f$.};

\coordinate(y_1) at (6,2);
\node at (5.3,1.7) {$1$};
\node at (6.75,2.3) {$2$};
\node at (6.1,0.4) {$3$};

\coordinate(y_2) at (7.5,2);
\node at (7.37,1.35) {$1$};
\node at (8.2,1.7) {$4$};
\node at (7.0,1.65) {$3$};
\node at (6.7,1.35) {$4$};

\coordinate(y_3) at (8.5,1);
\node at (8.2,0.4) {$3$};

\coordinate(y_4) at (7.5,0);
\node at (6.75,-0.3) {$2$};

\coordinate(y_5) at (6,0);
\node at (5.2,0.4) {$4$};

\coordinate(y_6) at (5,1);

\node at (5.6,0.78) {$2$};
\node at (6.6,0.45) {$1$};

%\node at (6,-1.5) {The graph $G$ with an edge coloring $g$.};

%\draw[ultra thick] (c)--(h);

\draw[thin] (x_1)--(x_2);

\draw[thin] (x_1)--(x_4);
\draw[thin] (x_1)--(x_5);
\draw[thin] (x_1)--(x_6);
\draw[thin] (x_2)--(x_3);
\draw[thin] (x_2)--(x_4);
\draw[thin] (x_2)--(x_6);
\draw[thin] (x_3)--(x_4);
\draw[thin] (x_3)--(x_5);
\draw[thin] (x_3)--(x_6);

\draw[thin] (x_4)--(x_5);

\draw[thin] (x_5)--(x_6);

\draw[thin] (y_1)--(y_2);
\draw[thin] (y_1)--(y_4);
\draw[thin] (y_1)--(y_5);
\draw[thin] (y_1)--(y_6);
\draw[thin] (y_2)--(y_3);
\draw[thin] (y_2)--(y_4);
\draw[thin] (y_2)--(y_6);
\draw[thin] (y_3)--(y_4);
\draw[thin] (y_3)--(y_5);
\draw[thin] (y_3)--(y_6);

\draw[thin] (y_4)--(y_5);
\draw[thin] (y_5)--(y_6);

\foreach \x in {x_1,x_2,x_3,x_4,x_5,x_6,y_1,y_2,y_3,y_4,y_5,y_6}
\fill (\x) circle (0.08);

\end{tikzpicture}
\end{center}
\caption{A graph with the edge colorings $f$ (to the left)
and $g$ (to the right).}
\end{figure}
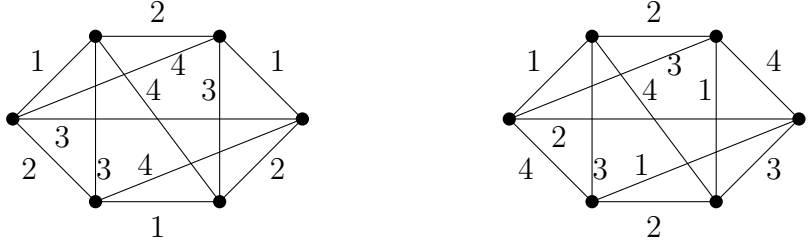	

	Consider the graph $G$ in Figure 1 with
	two different proper $4$-edge colorings $f$ 
	(to the left) and $g$ 
	(to the right).
%
%
%	Consider  the graph $G$ with vertices $x_1,...,x_6$ and  
%	its two proper 4--colorings $f$ and $g$ where
%$$M(f,1)=\{x_1x_5,x_2x_4,x_3x_6\}, M(f,2)=\{x_1x_4,x_2x_6,x_3x_5\},$$
%$$M(f,3)=\{x_1x_6,x_2x_3,x_4x_5\}, M(f,4)=\{x_1x_3,x_2x_5,x_4x_6\},$$
%$$ M(g,1)=\{x_1x_3,x_2x_6,x_4x_5\}, M(g,2)=\{x_1x_5,x_2x_3,x_4x_6\},$$
%$$ M(g,3)=\{x_1x_6,x_2x_4,x_3x_5\}, M(g,4)=\{x_1x_4,x_2x_5,x_3x_6\}.$$
	Clearly,
	$g$ cannot be obtained from $f$ by renaming the colors.
	On the other  hand any proper $3$-edge coloring of a
	subgraph $G(t_1,t_2,t_3)$ induced by the  set of
	edges $M(f,t_1)\cup M(f,t_2)\cup M(f,t_3)$ 
	gives the same partition of  edges of
	$G(t_1,t_2,t_3)$ into perfect matchings, for any $1\le t_1 < t_2 < 
	t_3\le 5$.
	This follows from the fact that $G(t_1,t_2,t_3)$ contains a
	triangle and  edges of this triangle belong to disjoint,
	uniquely defined perfect matchings.
	Therefore $g$ cannot be  obtained from $f$ even by
	transformations on 2-edge colored and  3-edge colored subgraphs.

	For a bipartite graph $G$ with $\Delta(G)=4$,
	any two proper $4$-edge colorings can always be transformed into each
	other by %using only
	recoloring 2-edge colored and 3-edge colored subgraphs.
	This follows from the result of  Asratian and Mirumian \cite{Asratian2} 
	(see also \cite{Asratian1}) that if $G$ is a bipartite graph
	with $\Delta(G)\geq 4$, then any proper 
	$\Delta(G)$-edge coloring of $G$ can be transformed 
	to any other proper $\Delta(G)$-edge coloring of $G$ in 
	such a way that each intermediate
	coloring is a proper $\Delta(G)$-edge coloring and differs 
	from the previous coloring by a
	$2$-- or $3$-edge colored subgraph.

\section{Preliminary results}

	First we introduce some terminology and notation.
	Let $\varphi$ be a proper $t$-edge coloring of $G$.
		For a vertex $v \in V(G)$,
	we say that a color $i$ {\em appears at $v$
	under $\varphi$} if there is an edge $e$ incident to $v$
	with $\varphi(e) = i$, and we set
	$$\varphi(v) = \{\varphi(e) : e \in E(G) \text{ and $e$ is incident to $v$}\}.$$

	If $\varphi$ is a proper $t$-edge coloring of $G$
	and $1 \leq a,b \leq t$, then a path or cycle in
	$G_\varphi(a,b)$ is called {\em $(a,b)$-colored under $\varphi$}.
	We also say that such a path or cycle is \emph{bicolored under $\varphi$}.
	A {\em $\varphi$-fan} at a vertex $v \in V(G)$
	is a sequence $(e_1,e_2,\dots,e_n)$ of edges
	such that 
	\begin{itemize}
	
	\item[(i)] the edges $e_1,e_2,\dots,e_n$ are distinct and incident to $v$,
	
	\item[(ii)] if $e_i = vu_i$,
	then $\varphi(e_i)$ does not appear at $u_{i-1}$, %under $\varphi$,
	$i=2,\dots, n$.
	
	\end{itemize}
	For such a $\varphi$-fan,
	we say that color $\varphi(e_i)$ is {\em associated with} $u_{i-1}$.
	A $\varphi$-fan $(e_1,e_2,\dots, e_k)$ at $v$ with $e_k = vu_k$
	is \emph{saturated} if there is a color $c$  
	such that $c \notin f(v) \cup f(u_k)$.
	For such a saturated $\varphi$-fan 
	$(e_1,e_2,\dots, e_k)$ we define a {\em downshift}
	as the following operation: define a new coloring $\varphi'$
	from $\varphi$ by setting $ \varphi'(e_k) = c$,
	$\varphi'(e_j) = \varphi(e_{j+1})$, $j=1,\dots,k-1$,
	and retaining the color of every other edge of $G$.
	Note that a downshift produces a new proper edge coloring
	and that it may be seen as a sequence of interchanges.
	
	In all the above definitions, we often leave out the explicit 
	reference to a
	coloring $\varphi$, if the coloring
	is clear from the context.
	The following three lemmas are fundamental for our proof of Theorem
	\ref{th:main}.

\begin{lemma}
\label{lem:onesided}
	Let $G$ be a $4$-regular graph,
	 $f$ be a proper $5$-edge coloring of $G$
	and assume that $P$ is a $(1,2)$-colored path under $f$
	in $G$.
	
\begin{itemize}
	\item[(a)] 
	Suppose that $P$ has length $4$ and the internal vertices
	along $P$ are $v_1,v_2,v_3$. Assume further that at least one
	of the following conditions does not hold:
	
	\begin{itemize}
	\item[(i)] $|f(v_i) \cap f(v_j) \setminus \{1,2\}|= 1$ for $i,j \in \{1,2,3\}$
	satisfying $i \neq j$,
	\item[(ii)] if $x_1$ and $x_2$
	are the vertices  distinct from $v_1$ and $v_3$
	that are adjacent to $v_2$,
	then colors $3,4,5$ appear at $x_1$ and $x_2$ under $f$.
	
	\end{itemize}

Then there is a proper $5$-edge coloring $g$  of $G$ which can be obtained from $f$ by a
sequence of interchanges on bicolored paths 
	with colors from $\{3,4,5\}$, such that
 one of colors $3,4,5$ is missing
	either in the set $ g(v_1) \cup g(v_2)$ or in $ g(v_2) \cup g(v_3)$.
	
	\item[(b)] 
If $P$ has length $5$ and $v_1,v_2,v_3, v_4$ are the
	internal vertices along $P$, then there is 
a proper $5$-edge coloring $g$  of $G$ which can be obtained from $f$ by a
sequence of interchanges on bicolored paths 
	with colors from $\{3,4,5\}$, such that one of colors $3,4,5$ is missing in
the set $ g(v_i) \cup g(v_{i+1})$, for
 some $i \in \{1,2,3\}$.
	
	\end{itemize}
\end{lemma}

\begin{proof}
	We first prove (a).  By the definition of the path 
	$P$, $\{1,2\} \subseteq f(v_i)$, for $i=1,2,3$.  
	We show that if the conclusion is false, then (i) and (ii) hold.
	So suppose that the conclusion of the lemma is false.
	Then $\{3,4,5\} \subseteq f(v_1) \cup f(v_2)$
	and $\{3,4,5\} \subseteq f(v_2) \cup f(v_3)$,  since otherwise
	the conclusion follows immediately by setting $g=f$. These two properties 
	together with the conditions $\Delta(G)=4$ and 
	$\{1,2\} \subseteq f(v_i)$, for 	$i=1,2,3$,  imply that  
	 $|f(v_1) \cap f(v_2) \setminus \{1,2\}| \leq 1$
	and
	$|f(v_2) \cap f(v_3) \setminus \{1,2\}| \leq 1$. 
	Since $G$ is $4$-regular, we thus have
	$|f(v_1) \cap f(v_2) \setminus \{1,2\}| = 1$ and
	$|f(v_2) \cap f(v_3) \setminus \{1,2\}| = 1$.

	We now prove that $|f(v_1) \cap f(v_3) \setminus \{1,2\}| = 1$.
	Suppose that this is not true, that is, 
	$|f(v_1) \cap f(v_3) \setminus \{1,2\}| = 2$.
	Without loss of generality we may assume that
	$f (v_1) = f(v_3) =\{1,2,3,5\}$. Then  $f(v_2) = \{1,2,4,5\}$ because 
	$\{3,4,5\} \subseteq f(v_2) \cup f(v_3)$.
	Consider a maximal path $Q$ that is $(3,4)$-colored 
	under $f$ and has origin at $v_2$.
	If $v_1$ is not an endpoint of this path, then color $4$
	neither appears at $v_1$ nor at $v_2$
	after an interchange on $Q$;
	if $v_1$ is an endpoint of $Q$, then color $4$
	neither appears at
	$v_2$ nor at $v_3$ after an interchange on $Q$. In both
	cases the conclusion of the lemma holds; a contradiction.
	Hence, $|f(v_1) \cap f(v_3) \setminus \{1,2\}| = 1$.  
	
	Suppose now that (i) holds but not (ii). 
	Without loss of generality we can assume that 
	$$f(v_1) = \{1,2,3,5\}, f(v_2) = \{1,2,4,5\}, f(v_3) = \{1,2,3,4\},$$
	and that $f(x_1v_2)= 5$. If color $3$ does not appear at $x_1$,
	then we may recolor the edge $v_2x_1$ with color $3$, obtaining a coloring
	where $5$ does not appear at $v_2$ or $v_3$;  a contradiction. 
	Similarly,   color
	$3$ must appear at $x_2$. Hence, $\{3,5\} \subseteq f(x_1)$
	and $\{3,4\} \subseteq f(x_2)$.
	
	Let us now prove that   color $4$ appears at $x_1$.
	Consider a maximal $(3,4)$-colored
	path $Q$ with origin $v_1$. If $v_2$ is not an endpoint
	of this path, then by interchanging colors
	on $Q$ we get a coloring
	where $3$ does not appear at $v_1$ or $v_2$.
	So suppose that $v_2$ is an endpoint of $Q$,
	interchange colors on $Q$ and denote the obtained coloring by $f_1$.
	If $4$ does not appear at $x_1$ under $f_1$, then we may recolor
	$v_2x_1$ with color $4$ to obtain a coloring where $5$ does not
	appear at $v_2$ or $v_3$. Hence,
	$4 \in f(x_1)$.
	That color $5$ must appear at $x_2$ under $f$ can be proved similarly by considering
	a maximal $(3,5)$-colored path with origin at $v_3$. 
 Thus $\{3,4,5\} \subseteq f(x_i)$, $i=1,2$.
	
	We now prove statement (b) of Lemma \ref{lem:onesided}.
	By part (a), we may assume that
	condition (i)  hold for the vertices
	$v_1, v_2, v_3$, since otherwise the required result 
	follows. In fact, we will as before
	assume (without loss of generality) 
	that $$f(v_1) = \{1,2,3,5\}, f(v_2) = \{1,2,4,5\}, f(v_3) = \{1,2,3,4\}.$$
	If $f(v_4) = \{1,2,3,4\}$, then the result follows by
	setting $g = f$.
	If $f(v_4) = \{1,2,4,5\}$, then the result follows
	by applying part (a) to the subpath of $P$ with 
	internal  vertices $v_2, v_3, v_4$. So we
	may assume that $f(v_4) = \{1,2,3,5\}$.
	Consider a maximal $(3,4)$-colored path $Q$ with origin at $v_2$.
	Let $f_1$ be a proper $5$-edge coloring obtained 
	from $f$ by interchanging  	colors on $Q$.
	If $v_1$ is not an endpoint of $Q$,  
	then the result follows by setting $g=f_1$ 	because 
	$4\notin f_1(v_1)\cup f_1(v_2)$. Suppose now that $v_1$ is an endpoint of $Q$.
	Then $|f_1(v_2)\cap f_1(v_4)\setminus \{1,2\}|=2$.
 	The result now follows
	by applying part (a) to the subpath of $P$ with internal vertices $v_2,v_3,v_4$.
\end{proof}

\begin{lemma}
\label{lem:x1x2}

	Let $G$ be a 4-regular graph with $\chi'(G)=4$,
	$f$ be a proper $5$-edge coloring of $G$, $h$
	a proper $4$-edge coloring of $G$,
	and let $P=u_1v_1v_2v_3v_4$ be a $(1,2)$-colored path under $f$
	in $G$
	such that $f(u_1v_1)=2$ and $h(u_1v_1)=1$. 
	Assume further that  
	$$f(v_1) = \{1,2,c_1,c_3\}, f(v_2) = \{1,2,c_2,c_3\}, 
	f(v_3) = \{1,2,c_1,c_2\},$$
	where $(c_1,c_2,c_3)$ is a permutation of the set $\{3,4,5\}$. 
	Let  $x_1$ and $x_2$ be the vertices  distinct from $v_1$ and $v_3$
	that are adjacent to $v_2$, where $f(v_2x_1)=c_3, f(v_2x_2)=c_2$. 
	If $u_1 \neq x_2$,  then one of the following holds:

\begin{itemize}	
	\item[(I)] $f(x_1)=\{2,3,4,5\}$ and $f(x_2)=\{1,3,4,5\}$,
	
	\item[(II)] there is a proper $5$-edge coloring $g$  of $G$ which can be obtained from $f$ by a
sequence of interchanges, such that $M(h,1) \cap M(g,1) = M(h,1) \cap M(f,1)$,
	the color $1$ is missing in $g(v_1)$
	and, moreover, if $1 \notin f(u_1)$, then $1 \notin g(u_1)$,

	\item[(III)] there is a proper $5$-edge coloring $g$  of $G$ which can be obtained from $f$ by a
sequence of interchanges on bicolored paths 
	with colors from $\{3,4,5\}$, such that
 one of colors $3,4,5$ is missing
	either in the set $ g(v_1) \cup g(v_2)$ or in $ g(v_2) \cup g(v_3)$.
	\end{itemize}
\end{lemma}

\begin{proof}
	Suppose that (III) does not hold. Then, by
 	Lemma \ref{lem:onesided},
	colors $3,4,5$ appear at both $x_1$ and $x_2$ under $f$. 

	First we show that $f(x_2)=\{1,3,4,5\}$ or (II) holds. 
	Suppose that $1\notin f(x_2)$. Then
	$v_1v_2x_2$ is a maximal $(1,c_2)$-colored path under $f$.
	Since, $u \neq x_2$ we may interchange colors on this path
	to obtain a proper $5$-edge coloring 
	$g$ of $G$ 
	for which (II) holds.
	Hence we may assume that  $1\in f(x_2)$,
  that is, $f(x_2)=\{1,3,4,5\}$.
	
	Now we show  that $f(x_1)=\{2,3,4,5\}$ or (II) holds.
	Suppose that $2\notin f(x_1)$.
	Then the sequence $(v_2v_1,v_2x_2,v_2v_3,v_2x_1)$
	is an $f$-fan at $v_2$ and color $2$ does not appear at
	$x_1$ and  $x_2$. Note also that $c_1$ does not appear at $v_2$.
	Consider a maximal path $Q$ that is $(c_1,2)$-colored under $f$ and has origin
	at $x_1$. 

	If $\{v_3,x_2\} \cap V(Q) =\emptyset$, 
	then we interchange colors on $Q$
	and obtain a proper coloring that we denote by $f_1$.
	The sequence $(v_2v_1,v_2x_2,v_2v_3,v_2x_1)$
	is a saturated $f_1$-fan at $v_2$ and we may downshift
	to obtain a proper $5$-edge coloring $g$ under which color $1$
	does not appear at $v_1$. Moreover,
	if $1 \notin f(u_1)$, then $1 \notin g(u_1)$; that is,
	(II) holds.
	
	Suppose now that $v_3 \in V(Q)$. Then $x_2 \notin V(Q)$.
	Consider a maximal path $Q'$ that is $(c_1,2)$-colored under $f$ and has
	origin $x_2$. Clearly, $Q$ and $Q'$ are disjoint
	and by interchanging colors on $Q'$ we get a coloring $f_1$
	where $(v_2v_1,v_2x_2)$ is a saturated $f_1$-fan at $v_2$.
	By downshifting on this fan 
	we obtain a coloring $g$ for which (II) holds.
	
	Suppose now that $x_2 \in V(Q)$. Then $v_3 \notin V(Q)$
	and by interchanging colors on $Q$ we get a coloring $f_1$
	where $(v_2v_1,v_2x_2)$ is a saturated $f_1$-fan.
	Again, downshifting on this fan yields that (II) holds.
\end{proof}	

	Let $G$ be a graph with $\chi'(G) = \Delta(G)=4$
	and $h$ a proper $4$-edge coloring of $G$.
	For an arbitrary  proper $5$-edge coloring $\varphi$ of $G$
	we say that an edge $e$ is {\em $h$-correct under $\varphi$}
	(or just {\em correct under $\varphi$}) if
	$\varphi(e) = h(e)$.
	By the {\em distance} between two edges $e$ and $e'$ on a path
	$P$, we mean the number of edges between $e$ and $e'$ on $P$.
	
\begin{lemma}
\label{lem:reduction}
	Let $G$ be a graph with $\chi'(G) = \Delta(G)=4$,
	$h$ be a proper $4$-edge coloring of $G$ and $f$ 
	a proper $5$-edge coloring
	of $G$. Furthermore, let $P$ be a maximal path or cycle
	that is $(1,2)$-colored under $f$
	and which
	contains an edge $xy$ such that $f(xy) =2$ and $h(xy) = 1$.
	If there is no $h$-correct edge of color $1$ under $f$
	at distance at most $2$ from $xy$ on $P$, then there
	is a proper $5$-edge coloring $f'$ of $G$  such that $f'$ 
	can be obtained from 
	$f$ by a sequence of interchanges and
	$|M(f',1) \cap M(h,1)| > |M(f,1) \cap M(h,1)|$.
\end{lemma}	

\begin{proof}
	Suppose first that $P$ contains no $h$-correct
	edge colored $1$.
	Then we simply make an interchange on $P$ to
	obtain the required coloring $f'$.
	
	Suppose now that $P$ contains some $h$-correct edge colored $1$.
	This implies that the length of $P$ is at least $6$.
	
\bigskip	

%%%%%%%%%%%%%%%%%%%%%%%%%%%%%%%%%%%%%%%%%%%%%%%

\noindent	
{\bf Case 1.}  $P$ is a path and the distance between $xy$ and the first or the last edge of $P$ is at  most 3.

\medskip
\noindent
Let $P = a_l a_{l-1} \dots a_1 x y b_1, \dots b_k$, where $l \leq 4$ or $k\leq 4$.
	Without loss of generality we assume that $l\leq 4$. The conditions imply that the subpath $a_l a_{l-1} \dots x$
	contains no $h$-correct edge colored $1$ under $f$, 
	so we must have $k \geq 5$.
	Consider the path $y b_1 b_2 b_3 b_4 b_5$. By Lemma 
	\ref{lem:onesided} (b),
	there is a sequence of
	interchanges on bicolored paths with colors from $\{3,4,5\}$
	yielding a proper $5$-edge coloring $g$ of $G$ from $f$, such that
	for some color $c \in \{3,4,5\}$ and some $i \in \{1,2,3\}$,
	we have $c \notin g(b_i) \cup g(b_{i+1})$.
	Note that $b_i b_{i+1}$ is not both $h$-correct and colored $1$.
	Recolor the edge $b_i b_{i+1}$ with the color $c$
	and denote the obtained coloring by $g_1$. Then the
	path $Q = a_l a_{l-1} \dots a_1 x y b_1 \dots b_i$ is
	a maximal $(1,2)$-colored path under $g_1$ and it
	contains no correct edge colored $1$, so the desired result follows
	by making an interchange on $Q$.

\bigskip	

\noindent	
{\bf Case 2.}  $P$ is a path and the distance between $xy$ and each of the first and last edges 
of $P$ is at least 4.

\medskip
\noindent
Let $P = a_l a_{l-1} \dots a_1 x y b_1, \dots b_k$,
	where $l,k \geq 5$. Consider the path $y b_1 b_2 b_3 b_4 b_5$. By Lemma 
	\ref{lem:onesided} (b),
	there is a sequence of
	interchanges on bicolored paths with colors from $\{3,4,5\}$
	yielding a proper $5$-edge coloring $g$ of $G$ from $f$, such that
	for some color $c \in \{3,4,5\}$ and  some $i \in \{1,2,3\}$,
	we have $c \notin g(b_i) \cup g(b_{i+1})$.
	Recolor the edge $b_i b_{i+1}$ with the color $c$
	and denote the obtained coloring by $g_1$.
	Then the 
	maximal path $Q$  that is
	$(1,2)$-colored under $g_1$ and contains  $xy$,
	satisfies the condition of Case 1. Therefore, the coloring $g_1$ can be transformed
to a required coloring $f'$ by a sequence of interchanges.
	Thus, the lemma is true in Case 2.

\bigskip	

\noindent	
{\bf Case 3.} $P$ is a cycle:

\medskip
\noindent
	Suppose that $P=a_k a_{k-1} \dots a_1 x y b_1 \dots b_k a_k$.
	If $P$ has length at most $8$, then it contains no
	$h$-correct edge colored $1$, which contradicts our assumption;
	so assume that $P$ has length at least $10$.
	If $P$ has length $10$, then there is
	at most one $h$-correct edge colored $1$ on $P$, 
	namely $b_4a_4$. Suppose that $b_4a_4$ is $h$-correct.
	By applying Lemma 
	\ref{lem:onesided} (b) to the path $y b_1 b_2 b_3 b_4 a_4$,
	and proceeding as above we may reduce this situation to Case
	1.
	
	If, on the
	other hand, $P$ has length at least $12$,
	then there is a path $y b_1 b_2 b_3 b_4 b_5$
	lying on $P$, and we may proceed exactly as in Case 2
	for reducing this case to Case 1.
\end{proof}

%%%%%%%%%%%%%%%%%%%%%%%%%%%%%%%%%%%%%%%%%%%%%%%%%

	\section{Proofs of Theorem \ref{th:connect} and Theorem \ref{th:connect1}}

In this section we prove Theorem \ref{th:connect} and Theorem \ref{th:connect1}.
First we present a family 
 of graphs permitting 
	an affirmative answer to Problem \ref{prob:vizing}.

	Denote by $G_\Delta$ the graph induced by the vertices of maximum
	degree in a graph $G$.
	It is known that if $G_\Delta$ is acyclic, then
	$G$ is a Class $1$ graph \cite{Fournier}.
	In fact, Fournier \cite{Fournier} 
	(see also \cite{BergeFournier})
		described an algorithm for sequentially 
	coloring the edges properly of such a graph using $\Delta(G)$ colors.

	Here, we use a similar  algorithm to prove
	the following proposition:

\begin{proposition}
\label{th:main2}
	Let $G$ be a graph where
	$G_\Delta$ is acyclic.  
	Then any proper $t$-edge coloring  of $G$ with $t > \Delta(G)$ can be transformed
	into a proper $\Delta(G)$-edge coloring by a sequence of interchanges.
\end{proposition}

\begin{proof}
	The major parts of the
	proof is similar to the proof of Vizing's theorem on the chromatic intex.
	Let $G$ be a graph with maximum degree $\Delta$ such that
	$G_\Delta$ is acyclic, $t$ an integer satisfying $t > \Delta$,
	and suppose that $f$ is a proper $t$-edge coloring of $G$.
	
	Taking Theorem A into consideration, it is sufficient
	to prove the proposition in the case $t = \Delta+1$.
	%, because if $t > \Delta+1$, 
	%then, by  Vizing's algorithm, $f$ can
	%be transformed into a proper $(\Delta+1)$-edge coloring
	%using only interchanges (see e.g. \cite{West});
	%the original problem now is reduced
	%to the problem of transforming the obtained
	%proper $(\Delta+1)$-edge coloring to 
	%a proper $\Delta$-edge coloring.
	%Thus we may assume that $t = \Delta+1$.

	Suppose that there is at least one edge $e$
	colored $\Delta +1$ in $G$. We will show that using 
	only interchanges we can transform
	$f$ to  another  proper $(\Delta+1)$-edge coloring with fewer edges
	colored $\Delta+1$.

\bigskip
\noindent
{\bf Case A.} There is an edge $e=vu_1 \in E(G_\Delta)$
with $f(e)=\Delta+1$ and  $d_{G_\Delta}(v) =1$:

\medskip
\noindent
		Since every neighbor $x \not= u_1$ of $v$ has degree at most $\Delta-1$,
	there is a maximal fan $F=(e_1,e_2, \dots, e_k)$
	at $v$ with $e_i = vu_i$ and with no vertex
	in $\{u_1, \dots, u_k\}$ being associated with color $\Delta+1$.
	Suppose that $f(e_i) =c_i, i=1,\dots,k$, where $c_1=\Delta+1$, 
	and that color $c_{k+1}\not=\Delta+1$ does not appear at $u_k$.
	
	Since $F$ is maximal, either 
	\begin{itemize}

	\item[(i)] there is no edge incident to $v$ colored $c_{k+1}$, or 
	
	\item[(ii)] $c_{k+1}=f(e_{j+1})$ for some $j, 1 \leq j<k-1$, that is,
	$c_{k+1}=c_{j+1}$ and $c_{k+1}$  does not appear at $u_j$.
	\end{itemize}
	In the 
	case (i) we downshift on $F$, and since no vertex in $\{u_1, \dots, u_k\}$ is
	associated with color $\Delta+1$, we obtain a proper $(\Delta+1)$-edge coloring
	with fewer edges colored $\Delta+1$.
	
	Suppose now that (ii) holds. Choose a color $c_0$ that does not appear at $v$.
	If $c_0$ does not appear also at $u_k$, then we replace $c_{k+1}$ 
	with $c_0$ and 
	proceed as in the case (i).
	Otherwise, we consider a maximal $(c_0,c_{k+1})$-colored path $Q$ 
	with origin at $u_k$.
	
	If $u_{j+1}$ is on $Q$, then $v$ is an endpoint of $Q$. 
	We interchange the colors
	on $Q$ and  denote the obtained coloring by $f_1$. 
	The color $c_{k+1}$ does not appear 
	now at vertices $v$ and $u_j$, so we may
	downshift on the saturated $f_1$-fan $(e_1,e_2,\dots, e_j)$
	to obtain the desired result.
	
	If $u_j$ is on $Q$, then this is an endpoint of $Q$.
	We interchange the colors on $Q$ and  obtain a coloring
	 $f_1$ where the color $c_0$ does not appear at vertices $v$ and $u_j$. 
	Now we downshift on the saturated $f_1$-fan $(e_1,\dots, e_j)$.
	
	If $\{u_{j-1}, u_j\} \cap V(Q) = \emptyset$, then we interchange colors on 
	$Q$ and then downshift on $(e_1, \dots, e_k)$.

\bigskip
\noindent
{\bf Case B.}
	There is no edge $e=uv \in E(G_\Delta)$ of color  $\Delta+1$ 
	with $d_{G_\Delta}(u) =1$
	or $d_{G_\Delta}(v) =1$.

\medskip
\noindent
	The following cases are possible:

\bigskip
\noindent
{\bf Case B.1} $E(G_\Delta) \cap M(f,\Delta+1) \neq \emptyset$:	

\medskip
\noindent
	Let $e = uv \in E(G_\Delta) \cap M(f,\Delta+1)$.
	We will describe an algorithm for constructing a sequence of distinct vertices
	$a_0, a_1,\dots,a_n$ in $G$ 
	such that
	$a_i \in V(G_\Delta)$, $i=0,\dots,n$, and
	$a_i$ and $a_{i+1}$ are adjacent in $G$.
	The algorithm will also construct a related 
	sequence $g_1, \dots g_{n-1}$
	of proper $(\Delta +1)$-edge colorings
	of $G$, such that the edge $a_ia_{i+1}$ is
	colored $\Delta+1$ under $g_i$.
	Additionally $|M(g_i,\Delta+1)| \leq |M(f,\Delta+1)|$
	for each $i \in \{1,\dots,n-1\}$.

\bigskip	
	
\noindent
{\bf Algorithm}

\bigskip	
\begin{flushleft}

\begin{description}	
	
	\item[Step $1$:] We first set $g_1 = f$ and $a_0 = u$
	and $ a_1=v$.

	\item[Step $i \, \, (i \geq 2)$:]
	Suppose that we have constructed the sequence
	$a_0,a_1,\dots,a_{i-1}$ of distinct vertices in $G$,
	such that $a_{j-1}$ and $a_{j}$ are adjacent, $j=1,\dots,i-1$,
	and a sequence of proper $(\Delta+1)$-edge colorings
	$g_1,\dots g_{i-1}$, satisfying that 
	$|M(g_j,\Delta+1)| \leq |M(f,\Delta+1)|$ 
	and $g_{j}(a_{j-1}a_{j}) = \Delta +1$, $j=1,\dots,i-1$.
	
	If $d_{G_\Delta}(a_{i-1}) = 1$, then  Stop. 
	The coloring $g_{i-1}$ satisfies the conditions in  Case A
	and we can   transform $g_{i-1}$ to a 
	proper $(\Delta+1)$-edge coloring
	with fewer edges colored $\Delta+1$.
	
	If $d_{G_\Delta}(a_{i-1})>1$, we
	consider a maximal $g_{i-1}$-fan 
	$F =(e_1,e_2, \dots, e_k)$
	at $a_{i-1}$ with $e_1=a_{i-1}a_{i-2}$ and $e_j = a_{i-1}u_j$, 
	for $j=2, \dots,k$.
	If there is some
	color $c_{k+1} \neq \Delta+1$ missing at $u_k$ under $g_{i-1}$, 
	then by 	proceeding
	similarly as in Case A we may 
	transform $g_{i-1}$ to
	a proper $(\Delta+1)$-edge coloring $\alpha$
	using only interchanges, where $\alpha$ satisfies
	$|M(\alpha,\Delta+1)| < |M(f,\Delta+1)|$.
	
	So suppose that $\Delta+1$ is the only color that does not appear at $u_k$
	under $g_{i-1}$.
	Note that this implies that $u_k \in V(G_\Delta)$
	Additionally, $u_k$ is distinct from $a_0,a_1,\dots, a_{i-1}$, because
	$u_k \neq a_{i-2}$ and
	$G_\Delta$ is acyclic.

	We will now make a series of interchanges from $g_{i-1}$
	to obtain $g_i$.
	Suppose that $g_{i-1}(e_j) = c_j$, $j=1,\dots,k$, where $c_1=\Delta+1$.
	We first interchange colors on the maximal
	$(\Delta+1,c_2)$-colored path with origin at $a_{i-2}$
	and denote the obtained coloring by $f_1$.
	Since the first edge of this path is colored
	$\Delta +1$, 
	$|M(f_1,\Delta+1)| \leq |M(g_{i-1},\Delta+1)|$.
	Now, for each $j =2,\dots,k-1$,
	from $f_{i-1}$ we construct the coloring $f_{i}$ by interchanging
	colors on the maximal $(\Delta +1,c_{i+1})$-colored path
	with origin at $u_i$. All these paths begin by an edge colored
	$\Delta +1$, so we do not increase the number of edges colored
	$\Delta+1$ in $G$ by making these interchanges.
	Note that  
	$f_{k-1}(a_{i-1}u_k)= \Delta+1$.
	We set $g_i=f_{k-1}$ and $a_i = u_k$, and go to Step $(i+1)$.

\end{description}
\end{flushleft}
	
	Since $G_\Delta$ is acyclic, the algorithm above will stop after
	a finite number of steps, when we reach a vertex of degree $1$
	in $G_\Delta$. Moreover, the algorithm
	produces a proper $(\Delta+1)$-edge coloring $g$ of $G$
	such that $|M(g,\Delta+1)| \leq |M(f,\Delta+1)|$,
	and there is an edge $e \in E(G_\Delta) \cap M(g,\Delta+1)$ 
	which is incident to
	a vertex of degree $1$ in $G_\Delta$.
	Hence, this case has been reduced to Case A, and thus we may,
	using only interchanges,
	from $g$ construct
	a proper $(\Delta+1)$-edge coloring
	with fewer edges colored $\Delta+1$.

\bigskip
\noindent
{\bf Case B.2} $E(G_\Delta) \cap M(f,\Delta+1) = \emptyset$
but some edge colored $\Delta+1$ is incident to a vertex of degree $\Delta$:

\medskip
\noindent
	Let $e_1 = u_1v$ be an edge colored $\Delta +1$
	and suppose that $d_G(v) = \Delta$.
	We consider a maximal fan 
	$F =(e_1, e_2, \dots, e_k)$,
	at $v$ with $e_i = vu_i$, for $i=1,\dots,k$. If there
	is a color $c \neq \Delta+1$ missing at $u_k$, then
	we proceed as in Case A 
	and construct a proper $(\Delta+1)$-edge coloring $\alpha$ from $f$ via interchanges,
	such that $|M(\alpha,\Delta+1)| < |M(f,\Delta+1)|$.
	
	So suppose $\Delta+1$ is the only color missing at $u_k$ under $f$.
	This implies that $u_k \in V(G_\Delta)$.
	We may now proceed exactly as described in Step $i$ of the algorithm above:
	by making a series of interchanges on maximal bicolored paths (where $\Delta+1$
	is one of the colors on every such path),
	we obtain a proper $(\Delta+1)$-edge coloring $g$ of $G$
	such that $|M(g,\Delta+1)| \leq |M(f,\Delta+1)|$
	and $g(vu_k) =\Delta+1$; that is,
	there is an edge colored $\Delta +1$ under $g$ both ends of which have
	degree $\Delta$. Then the edge $e=vu_k$ satisfies the conditions  of 
	Case A or Case B.1; as before, we can construct a
	proper $(\Delta+1)$-edge coloring
	with fewer edges colored $\Delta+1$ than under $g$.

\bigskip
\noindent
{\bf Case B.3} $E(G_\Delta) \cap M(f,\Delta+1) = \emptyset$ and
no edge colored $\Delta+1$ is incident to a vertex of degree $\Delta$:

\medskip
\noindent
	Suppose that $e_1 = u_1v$ is colored $\Delta +1$ under $f$,
	and consider a maximal fan 
	$F =(e_1, e_2, \dots, e_k)$
	at $v$, where $e_i = vu_i$, for $i=1,\dots,k$. 
	If there
	is a color $c \neq \Delta+1$ missing at $u_k$, then
	we proceed as in Case A 
	and construct a proper $(\Delta+1)$-edge coloring 
	$\alpha$ from $f$ via 	interchanges,
	such that $|M(\alpha,\Delta+1)| < |M(f,\Delta+1)|$.
	
	So assume that
	$\Delta+1$ is the only color missing at $u_k$.
	Then $u_k \in V(G_\Delta)$ and
	we may now proceed as in
	Step $i$ of the algorithm above for obtaining
	a proper $(\Delta+1)$-edge coloring
	$g$ of $G$ from $f$, such that
	$|M(g,\Delta+1)| \leq |M(f,\Delta+1)|$ and there is an
	edge $e$ colored $\Delta +1$ under $g$ such that one end of $e$
	is $u_k$, and thus
	has degree $\Delta$. Hence, this case has been reduced to Case B.2.
	This completes	the proof of Proposition \ref{th:main2}.
\end{proof}

\begin{proof}
[{\bf Proof of Theorem	\ref{th:connect}.}]

 	If for every graph $G$ 
 	with $\Delta(G) \leq k$
 	all proper $(\chi'(G)+1)$-edge colorings
 	of $G$ are Kempe equivalent, then,
	using Theorem A, it is not difficult to show
	that every Class 1 graph $H$ with $\Delta(H) \leq k$
	has the Vizing property.

	Conversely, assume  that  all Class 1 graphs with 
	maximum degree at most $k$ have 	the Vizing property.
	We will prove  that for each graph $G$ with $\Delta(G)\leq k$ all proper 
	$(\chi'(G)+1)$-edge colorings of $G$ are Kempe equivalent. 
	The proof is by induction on $k$.

	The proposition is evident  for $k=2$. Suppose that the 
	proposition  is true for all integers less than $k$,
 	and let $G$ be an arbitrary graph with $\Delta(G)=k$ and $\chi'(G)=\chi'$.
	Consider  a proper $\chi'$-edge coloring $h$ of $G$.
	We assume that $M(h,\chi')$ is a maximal matching of $G$. 
	Otherwise we can 	consider instead of $h$ another proper 
	$\chi'$-edge coloring which can be obtained from $h$   
	by a sequence of interchanges as follows: if there is an 
	edge $xy$ such that the color $\chi'$ is missing at $x$ and at $y$, 
	then recolor the edge $xy$ with color $\chi'$. 
	We repeat this procedure until we obtain a proper 
	$\chi'$-edge coloring where 	the edges colored $\chi'$ form a maximal matching.

	In order to show that any two proper $(\chi'+1)$-edge 
	colorings of $G$ are Kempe equivalent
 	it is sufficient to show that any proper $(\chi'+1)$-edge coloring  
 	of $G$ can be transformed to $h$
 	by a sequence of interchanges. 

	Let $\varphi$ be a proper $(\chi'+1)$-edge coloring of $G$.
	First we   transform $\varphi$ by a sequence of interchanges 
	to a proper 	$\chi'$-edge coloring. 
	In the case when  $G$ is a Class 1 graph it is possible because 
	$G$ has the Vizing property, and in the case when
	$G$ is a Class 2 graph it is possible by Theorem A. 
	The obtained coloring we denote by $\varphi_1$.

	Then the coloring $\varphi_1$ can be transformed  to a proper
	$(\chi'+1)$-edge coloring  by sequentially recoloring the edges in the matching  	$M(h,\chi')$  with color $\chi'+1$. Clearly,  
	the obtained proper 	$(\chi'+1)$-coloring $f$ satisfies the 
	condition $M(h,\chi')=M(f,\chi'+1)$. Furthermore,
	$f$ is obtained from $\varphi$ by a sequence of interchanges.

	Consider the graph $G_1=G-M(h,\chi')$.  
	The proper $\chi'(G)$-edge coloring $h$ of $G$ induces a proper 
	$(\chi'(G)-1)$-edge  coloring $h'$ of $G_1$ and the proper 
	$(\chi'(G)+1)$-edge 	coloring $f$ of $G$ induces a proper 
	$\chi'(G)$-edge  coloring $f'$ of $G_1$. Clearly, $\chi'(G_1)=\chi'(G)-1$.
	Therefore, $h'$ is a proper $\chi'(G_1)$-edge coloring and $f'$ is a proper
	$(\chi'(G_1)+1)$-edge coloring of $G_1$.

	If $\Delta(G_1)=k-1$ then, by the induction hypothesis, $f'$ and $h'$ are Kempe 	equivalent.
	Then the coloring $f$  of $G$ can be transformed to the 
	coloring $h$ by a 		sequence of interchanges.

	Suppose now that   $\Delta(G_1)=k$.  This is possible 
	only if  $G$ is  a Class 2 graph and $G_1$ is a Class 1 graph, 
	that is, if $\chi'(G)=k+1$ and $\chi'(G_1)=k$.
	Since $M(h,k+1)$ 
	is a maximal matching in $G$, every two vertices of degree 
	$k$ in $G_1$ are  non-adjacent.  
	By Proposition  \ref{th:main2} the coloring $f'$ 
	can be transformed by interchanges to a proper
	$(k+1)$-edge coloring $f''$ such that $M(f'',k+1)=\emptyset$.  
	Let $g$ be the   	proper 
	$(k+1)$-edge coloring obtained from $f''$  by sequentially
 	recoloring the edges in the set
	$E_0=M(h', k)$ with color $k+1$.    Clearly, $g$ is obtained from
	$f''$ by a sequence of interchanges and $M(g,k+1)=M(h',k)$.
	Then the graph $G_2=G_1-M(h',k)$  has the maximum degree $k-1$.
  Let
	$h_1$ be the  proper $(k-1)$-edge coloring of
	$G_2$ induced by the matchings $M(h,1), M(h,2),\dots,M(h,k-1)$  and $g_1$
	be the  proper $k$-edge coloring of
	$G_2$ induced by the matchings $M(g,1),M(g,2),\dots,M(g,k)$. 
	By the induction hypothesis,
  $g_1$ can be transformed to $h_1$ by a sequence of interchanges.
	Then the coloring $f$ also can be transformed to $h$ 
	by a sequence of 	interchanges.

	We have thus proved that for every graph $G$ with $\Delta(G)=k$ 
	all proper 	$(\chi'(G)+1)$-edge colorings of $G$ are Kempe equivalent.

	The proposition follows by the principle of induction. 
\end{proof}

	Next, we will prove	Theorem	\ref{th:connect1}.
	The proof is similar to the proof of Theorem \ref{th:connect}
	and therefore we only sketch the proof.

\begin{proof}[{\bf Proof of Theorem	\ref{th:connect1}}]
Assume that for every  graph $H$ with $\Delta(H)=k-1$ all proper $(\chi'(H)+1)$-edge colorings of $H$ are Kempe equivalent,
and let $G$ be  a Class 2 graph with $\Delta(G)=k$ and $\chi'(G)=k+1$.

	Consider  a proper $(k+1)$-edge coloring  $h$ of $G$.  Again as in the proof of 	Theorem \ref{th:connect} we may assume that 
	$M(h,k+1)$ is a maximal matching of $G$. Our goal is then to show 
	that any proper 	$(k+2)$-edge coloring of $G$ 
	can be transformed to $h$ by a sequence of interchanges. The proof is the same 
 	as in the proof of Theorem \ref{th:connect} with only one difference: 
 	instead of the induction hypothesis we  use the above assumption 
 	of  the theorem.
 %
 %This completes the proof of Theorem \ref{th:connect1}.
\end{proof}

%%%%%%%%%%%%%%%%%%%%%%%%%%%%%%%%%%%%%%%%%%%%%%%%%

   \section{Proof of Theorem \ref{th:main}}

	In this section we will prove the main result of our paper.
First we prove the result for 4-regular graphs and then, using it, for any graph  with maximum degree 4.\medskip

\begin{theorem}
\label{th:prel}	
	Let $G$ be a  $4$-regular graph with $\chi'(G)=4$. Then every proper
	$5$-edge coloring of $G$ 
	can be transformed to any proper  $4$-edge coloring  of $G$  
	by a sequence of 	interchanges. 
\end{theorem}

\begin{proof}	
Let $h$ be a proper  $4$-edge coloring  and $f$ be a proper  $5$-edge coloring  of the graph $G$.
	It suffices to prove that  the coloring
	$f$ can, by a sequence of 
	interchanges, be transformed to  a proper $5$-edge coloring
	$g$ such that	$M(g,1) = M(h,1)$. The result then
	follows by applying Theorem B  to the graph $G'=G-M(g,1)$ with maximum degree 3
	and two proper edge colorings of $G'$ induced 
	by the colorings $g$ and $h$, respectively.
		
	Suppose that there is some edge $e = u_1v_1$ of $G$
	such that $h(e) = 1 \neq f(e)$.  
	Without loss of generality we assume that $f(e)=2$.
	We will prove
	that  by a sequence of interchanges we can transform $f$ into
	a proper $5$-edge coloring $f'$ which satisfies 

\begin{equation}
|M(f',1) \cap M(h,1)| > |M(f,1) \cap M(h,1)|.
\label{E1}
\end{equation}	
This suffices for proving the theorem.
	We will have to consider several cases.
	
\bigskip

%%%%%%%%%%%%%%%%%%%%%%%%%%%%%%%%%%%%%%%%%%%%%%%%%%%%%

\noindent
{\bf Case A.} $e$ is adjacent to at most one edge colored $1$ under $f$:

\medskip
\noindent	
	Without loss of generality we assume that
	color $1$ appears at 
	$v_1$ under $f$ (if $e$ is not adjacent to any edge colored $1$,
	then we just recolor $e$ with $1$ and are done). Consider
	a maximal path $P = u_1v_1v_2 \dots v_k$ in 
	$G$ that is $(1,2)$-colored and has origin at $u_1$.
	Note that since $h(u_1v_1) =1$, $v_1v_2$
	is not correct under $f$.
	If $k \leq 3$, then, by Lemma \ref{lem:reduction},
	we can transform $f$ by interchanges to a required coloring 
	$f'$ satisfying \eqref{E1}.
	Suppose now that $k \geq 4$. 
	If $v_3v_4$ is not correct, then 
	again by Lemma \ref{lem:reduction},
	we can transform $f$ by interchanges to a required coloring 
	$f'$ satisfying \eqref{E1}. So
	in the following we will assume that $v_3v_4$ is correct.

\medskip
\noindent
	{\bf Case A.1.}  One of the conditions (i), (ii) in Lemma
	\ref{lem:onesided}  does not hold.

\medskip
\noindent	
	Then
	there is a sequence of interchanges on bicolored paths 
	with colors from $\{3,4,5\}$ which transform $f$ 
	to a proper $5$-edge coloring $g$  such that 
	$c \notin g(v_i) \cup g(v_{i+1})$
	for some color $c \in \{3,4,5\}$ and some $i\in \{1,2\}$.
	Now we  recolor $v_iv_{i+1}$ with $c$ and 
	 denote  by $g_1$ the proper edge coloring obtained from this operation. 
	Then the maximal $(1,2)$-colored path (under $g_1$) with origin at $u_1$ 
	has length at most two and, by
	Lemma \ref{lem:reduction}, $g_1$ can  by   interchanges be transformed
	into a proper 5-edge coloring $f'$ such that $|M(f',1) \cap M(h,1)| > |M(g_1,1) 	\cap M(h,1)|$.
	This and $M(g_1,1) \cap M(h,1) = M(f,1) \cap M(h,1)$ imply that
	$|M(f',1) \cap M(h,1)| > |M(f,1) \cap M(h,1)|$. 

\medskip
\noindent
	{\bf Case A.2.} The conditions (i),(ii)
	in Lemma \ref{lem:onesided} hold.

\medskip
\noindent
	This means that %$d_G(v_i) = 4$, for $i=1,2,3$;
	$|f(v_i) \cap f(v_j) \setminus \{1,2\}|= 1$ for $i,j \in \{1,2,3\}$
	and $i \neq j$; and the
	colors $3,4,5$ appear at $x_1$ and $x_2$, 
where $x_1,x_2$ denote the vertices  distinct from $v_1$ and $v_3$
	that are adjacent to $v_2$.
 Then
$$f(v_1) =\{1,2,c_1,c_3\}, f(v_2) = \{1,2,c_2,c_3\},
	f(v_3) =\{1,2,c_1,c_2\}$$
for some permutation $ (c_1,c_2,c_3)$ of the set $\{3,4,5\}$.
Clearly, $\{f(v_2x_1),f(v_2x_2)\}=\{c_2,c_3\}$.
 Let $f(v_2x_1)=c_3$ and $f(v_2x_2)=c_2$. 
Without loss of generality
	we  may assume that
	$$f(v_1) =\{1,2,3,5\}, f(v_2) = \{1,2,4,5\},
	f(v_3) =\{1,2,3,4\}.$$
Then $f(v_2x_1)=5$ and $f(v_2x_2)=4$.

\medskip
\noindent
{\bf Case A.2.1.} The vertex $v_4$ is an endpoint of $P$ and $x_2\not=u_1$.

\medskip
\noindent
	First suppose  that  either  $f(x_1)\not=\{2,3,4,5\}$ 
	or $f(x_2)\not=\{1,3,4,5\}$.
	Then, by Lemma \ref{lem:x1x2},  either  condition (II)
	or condition (III) of Lemma  \ref{lem:x1x2} holds. 
	
If condition
	(III)   holds,
	then arguing similarly as in Case A.1  we
	can by interchanges transform  $f$  into a  required coloring $f'$
	satisfying \eqref{E1}.
	
	If condition
	(II)   holds, then %by Lemma \ref{lem:x1x2}
	 we can  by interchanges transform  $f$
into a proper 5-edge coloring $g$
	such that $M(g,1) \cap M(h,1) = M(f,1) \cap M(h,1)$. Moreover,
	since color $1$ does not appear at $u_1$ under $f$, we will have $1\notin g(v_1), 1\notin g(u_1)$. We color the edge $u_1v_1$ with color 1 and obtain a required coloring $f'$ satisfying \eqref{E1}.
		
	Now suppose  that $f(x_1)=\{2,3,4,5\}$ and 
	 $f(x_2)=\{1,3,4,5\}$. 
	 
	We first consider the case
	when $u_1\not=x_1$. We have that $f(v_2x_1)=5$.
	Clearly, $x_1 \neq v_4$ because $1\in f(v_4)$ and 
	$f(x_1)=\{2,3,4,5\}.$
	We make an interchange on $P$ and denote the obtained coloring by $f_1$.
	Then $u_1v_1$ is correct under $f_1$, but $v_3v_4$ is not.
	Moreover, $x_1 v_2 v_3$ is a maximal $(5,1)$-colored path under $f_1$,
	and by interchanging colors on this path we get a new coloring which
	we denote by $f_2$.
	Since $x_1 \neq v_4$, we may now recolor $v_3 v_4$ with color $1$
	to obtain a required  coloring $f'$.
	
	Suppose now that $u_1 = x_1$.
	We have that
	$f(x_2) = \{1,3,4,5\}$ and $d_G(v_2)=4$. Then $v_2$ is incident
	to an edge colored $1$ under $h$. Furthermore, $$h(x_1v_1)=h(v_3v_4)=1,$$
	and thus we must have $h(v_2x_2)=1$.
	Let $x_3$ be the vertex adjacent to $x_2$, such that $f(x_2x_3)=1$.
	Since $h(v_2x_2)=1$, $x_2x_3$ is not correct under $f$, 
	which implies that $v_4 \notin \{x_2,x_3\}$.
	We now interchange colors on $P$ and denote the obtained coloring by $f_1$. 
	Then $x_1v_1$ is correct under $f_1$, but $v_3v_4$ is not.
	Note that $v_1 v_2 x_2$ is a maximal $(4,2)$-colored path under $f_1$.
	We make an interchange on this path and denote the 
	obtained coloring by $f_2$.
	Let $P' = v_4 v_3 v_2 x_2 x_3 \dots$ be
	the maximal $(1,2)$-colored path under $f_2$ with origin at $v_4$.
	Since, $x_2x_3$ is not correct under $f$, it is not correct
	under $f_2$ too, and the desired result now follows  by applying 
	Lemma \ref{lem:reduction} to the coloring $f_2$, the path 
	$P'$ and the edge $v_4v_3$
	instead of $f, P$ and $xy$.

\medskip
\noindent
{\bf Case A.2.2.} The vertex $v_4$ is an endpoint of $P$ and $x_2=u_1$.

\medskip
\noindent
Recall that
	colors $3,4,5$ appear at $x_1$ under $f$.
	If $1$ appears at $x_1$ under $f$, then the edge colored
	$1$ incident to $x_1$ cannot be correct under $f$,
	because we must have $h(v_2x_1)=1$.
	This implies that $v_4 \neq x_1$.
	Furthermore, $x_1 v_2 v_3$
	is a maximal $(5,2)$-colored path under $f$ and by making an interchange
	on this path, we obtain a coloring that we denote by $f_1$. Now, note that
	$P' = x_2 v_1 v_2 x_1 \dots$ is a maximal $(1,2)$-colored path under $f_1$,
	and since the edge incident to $x_1$ colored $1$ is not correct,
	it now follows from Lemma \ref{lem:reduction}
	that there
	is a proper $5$-edge coloring $f'$ of $G$ that can be obtained from
	$f_1$ via a sequence of interchanges, and such that
	$|M(f',1) \cap M(h,1)| > |M(f,1) \cap M(h,1)|$.
	If instead the color  $2$ appears at $x_1$ under $f$, then the proof proceeds exactly
	as in  Case A.2.1 when $u_1 \not=x_1$.

\bigskip
\noindent
{\bf Case A.2.3.} The vertex $v_4$ is an internal vertex of $P$.

\medskip
\noindent
It follows from Lemma 2.1 (b) that
there is a sequence of interchanges on bicolored paths 
	with colors from $\{3,4,5\}$ which transform $f$ 
	to a proper $5$-edge coloring $g$  such that 
$c \notin g(v_i) \cup g(v_{i+1})$
	for some color $c \in \{3,4,5\}$ and some $i\in \{1,2,3\}$.
	
	If $i\leq 2$, then arguing similarly as in Case A.1, we can obtain a required coloring $f'$
satisfying \eqref{E1}.	
	Now we suppose that such a color $c$ exists only for $i=3$. This means that  $c\notin g(v_3)\cup g(v_4)$,
 $\{3,4,5\}\subseteq g(v_1)\cup g(v_2)$ and $\{3,4,5\}\subseteq g(v_1)\cup g(v_2)$.
	
	If the coloring $g$ does not satisfy one of the conditions (i) and  (ii) in Lemma 2.1 (with $g$ instead of $f$), then in a similar way as in 
	Case A.1 we can obtain a required coloring $f'$.
	
	Suppose now that the conditions (i), (ii) of Lemma 2.1 
	for the coloring $g$ 	(instead of $f$)  hold. Then  
	$$g(v_1) =\{1,2,b_1,b_3\}, g(v_2) = \{1,2,b_2,b_3\},
	g(v_3) =\{1,2,b_1,b_2\}$$
	for some permutation $ (b_1,b_2,b_3)$ of the set $\{3,4,5\}$.
	Clearly, $\{f(v_2x_1),f(v_2x_2)\}=\{b_2,b_3\}$.

	Without loss of generality we can assume that
	$$g(v_1) = \{1,2,3,5\}, g(v_2) = \{1,2,4,5\}, g(v_3) = \{1,2,3,4\},$$
 	and that $g(x_1v_2) = 5$, $g(x_2v_2)=4$
 	(possibly renaming the vertices $x_1$ and $x_2$ if this does not hold).
 	Since there is a color $c\notin g(v_3)\cup g(v_4)$ and $g(v_3) = \{1,2,3,4\}$,
	we have that $c=5$ and  $g(v_4)=\{1,2,3,4\}$. 
	If $u_1 \neq x_2$, then similarly as in Case A.2.1,
	it follows from 
	Lemma 2.2 that either
	$g(x_1)=\{2,3,4,5\}$ and $g(x_2)=\{1,3,4,5\}$, or 
	we can obtain a required coloring $f'$ satisfying
	\eqref{E1}.
	Thus if $u_1 \neq x_2$, then it suffices to consider the case
	when $g(x_1)=\{2,3,4,5\}$ and $g(x_2)=\{1,3,4,5\}$.
% 	Again using Lemma 2.1(a)
% 	we may assume that colors $3,4,5$ appear at
% 	$x_1$ and $x_2$ under $g$, and so they are not internal vertices of $P$.
 	%

 	We first consider the case when $u_1 \notin \{x_1,x_2\}$.
 	From $g$ we costruct the coloring $g_1$ by recoloring $v_3v_4$ with color $5$.
 	Then
 	$g_1(x_1)=\{2,3,4,5\}$ and $g_1(x_2)=\{1,3,4,5\}$.
 	Now
 	we make an interchange on the path $u_1 v_1 v_2 v_3$ and denote the constructed
 	coloring by $g_2$. Note that $v_4 v_3 v_2 x_1$ is a maximal
 	$(1,5)$-colored path under $g_2$. By making an interchange on this path
 	we obtain a required coloring $f'$ satisfying \eqref{E1}.

 	Now assume that $u_1=x_1$.
 	From $g$ we construct the coloring $g_1$ by recoloring $v_3v_4$ with color $5$.
 	Then $g_1(x_1)=\{2,3,4,5\}$ and $g_1(x_2)=\{1,3,4,5\}$.
 	By making an interchange on the path $x_1v_1v_2v_3$ we obtain a coloring
 	$g_2$. Note that the path $v_4 v_3 v_2 x_1 v_1$ is $(5,1)$-colored 
 	under $g_2$
 	and that the edge $v_3 v_4$ is adjacent to only one edge colored $1$.
 	Moreover, $v_2$ and $v_4$ are not adjacent, because
 	$g(x_1) \neq g(v_4)$ and $g(x_2) \neq g(v_4)$.
	Consider the path $P'=u_1'v_1'v_2'v_3'v_4'$, where 
	$u_1'=v_4, v_1'=v_3, v_2'=v_2, v_3'=u_1$ and $v_4'=v_1$.
 	 If we exchange   the colors $5$ and $2$ in the coloring $g_2$ 
	and  consider the obtained coloring instead of $g$, 
	the path 	$P'=u_1'v_1'v_2'v_3'v_4'$
	instead of $P$ and the vertices $x_1'=v_1, x_2'=x_2$ 
	instead of $x_1$ and  $x_2$,  then we
	obtain a situation which is similar to the situation  considered in the
	previous paragraph, because here $u_1' \notin \{x_1',x_2'\}$. 
 		
	Now assume that $u_1=x_2$. 
 	Recall that
	colors $3,4,5$ appear at both $x_1$ and $x_2$ under $f$.
 	As before, we construct
 	the coloring $g_1$ from $g$ by recoloring $v_3v_4$ with color $5$.
 	Thereafter, we make a interchange
 	on the path $x_2 v_1 v_2 v_3$ and denote the obtained coloring by $g_2$.
    
 	If $x_1$ is incident to an edge colored $2$ under $g_2$,
	then $v_4 v_3 v_2 x_1$ is a maximal $(1,5)$-colored path under $g_2$.
	By  making an interchange on this path we obtain a
	required coloring $f'$ satisfying \eqref{E1}. 

 	Suppose now that  $x_1$ is incident to an edge colored $1$ under $g_2$. 
 	Consider the path
 	$v_4 v_3 v_2 x_1$ that is $(5,1)$-colored path under $g_2$ with origin at
 	$v_4$. Clearly, the edge colored $1$ incident to $x_1$ 
 	cannot be correct under $g_2$ (because we must have $h(x_1v_2)=1$).
	Put $x'=v_4$, $y'=v_3$ and let $P'=v_4 v_3 v_2 x_1 \dots$ be 
 	the maximal $(5,1)$-colored path with origin at $x'=v_4$.
	Now we obtain the required result by applying Lemma \ref{lem:reduction}
	to the coloring $g_2$, 
	the path $P'$ and the edge $x'y'$ instead of $f, P$ and $xy$.

\medskip
\noindent

\bigskip
\noindent
{\bf Case B.} $e$ is adjacent to two edges colored $1$ under $f$.

\medskip
\noindent
	We have that $f(e) =2$.  Let $u_1u_2$ and $v_1v_2$ 
	be the edges adjacent to $e$ 	and colored $1$ under $f$. Since $h(e) =1$, 
	none of the  edges $u_1u_2$ and $v_1v_2$  are $h$-correct.
 	Let $P$
 	be the component in $G_f(1,2)$ that contains $e$. It can be a path or cycle.

	If there is no correct edge colored $1$ under $f$ at distance $2$
	from $e$ along $P$, then the theorem is true by
	Lemma \ref{lem:reduction}. 

	Suppose now that there is
	a correct edge at distance $2$ from $e$ on $P$.
	This implies that if $P$ is a cycle then it has the length at least 6.

\bigskip
\noindent
{\bf Case B.1.} $P$ is a path where the distance between $e$ and the first or the last edge of $P$ 
 is at most 1.
\medskip

\noindent 	
Let  $P=u_lu_{l-1}...u_1v_1v_2...v_k$, where  $2\leq l \leq 3$ or $2\leq k\leq 3$.
Without loss of generality we may assume that $2\leq l\leq 3$. Since there is a correct edge at distance 2 from $e$ on $P$, we will have that $k\geq 4$ and  the edge
	$v_3v_4$ is correct. If one of the conditions 
	(i) and (ii) of Lemma \ref{lem:onesided} does not hold,  then arguing similarly as in Case A.1  we
	can by a sequence of interchanges transform  $f$  to a  proper coloring $f'$ satisfying \eqref{E1}.

	Assume now that the conditions
	(i), (ii) of Lemma \ref{lem:onesided}  hold.
	Then %arguing similarly as in Case A.1  
	we may assume  without loss of generality  that
	$$f(v_1) =\{1,2,3,5\}, f(v_2) = \{1,2,4,5\},
	f(v_3) =\{1,2,3,4\}, f(v_2x_1)=5, f(v_2x_2)=4,$$
	and the colors
	$3,4,5$ appear at $x_1$ and $x_2$,  where $x_1,x_2$  are the vertices  distinct 	from $v_1$ and $v_3$ that are adjacent to $v_2$.  Since $l\geq 2$, we have that 	$\{1,2,\}\subseteq f(u_1)$. Therefore $u_1 \notin \{x_1,x_2\}$.

	Suppose that $f(x_1) \neq \{2,3,4,5\}$.
	Then one of the conditions (II) and (III) of Lemma \ref{lem:x1x2} holds.
	If the condition (II) holds,  then we obtain a proper 5-edge 
	coloring $g$ such 	that the color 1 is missing at $v_1$.
	This situation for $g$ is similar to 
	the situation considered in Case A for $f$. Therefore we can transform $g$ by a 	sequence of interchanges to a required coloring $f'$ satisfying \eqref{E1}.

	If the condition (III) holds, then 
	there is a sequence of interchanges on bicolored paths 
	with colors from $\{3,4,5\}$ which transform $f$ 
	to a proper $5$-edge coloring $g$  such that 
	$c \notin g(v_i) \cup g(v_{i+1})$
	for some color $c \in \{3,4,5\}$ and some $i\in \{1,2\}$.
	Now we  recolor $v_iv_{i+1}$ with $c$ and 
	denote  by $g_1$ the proper edge coloring obtained from this operation. 
	Then the path $P'=u_lu_{l-1}...u_1v_1v_2...v_i$ is the maximal $(1,2)$-colored (under $g_1$) path with origin at $u_l$, such that 
	$g_1(u_1v_1)=2, h(u_1v_1)=1$ and there is no 
	$h$-correct edge of color 1 on $P'$.
	This situation  for $g_1$ is similar to 
	the situation considered in Lemma \ref{lem:reduction} for 
	the coloring $f$, the path $P$ and the edge $xy$.
	Therefore we can transform $g_1$ 
	by a sequence of interchanges to a required coloring $f'$
satisfying \eqref{E1}.

		We now consider the case when $f(x_1)=\{2,3,4,5\}$.
	%In particular, this implies that
	%$v_4 \neq x_1$.
	%
	If $v_4$ is an endpoint of $P$ then we proceed as follows:
	make an interchange on $P$ and denote the obtained coloring by $f_1$.
	Note that $|M(f_1,1) \cap M(h,1)| = |M(f,1) \cap M(h,1)|$, since
	$u_1v_1$ is correct under $f_1$, but $v_3v_4$ is not.
	However, $v_3v_4$ is adjacent to only one edge colored
	$1$ under $f_1$. Thus we have a situation that 
	is similar to the situation 	considered in Case A
	with the coloring $f_1$ and the edge $v_3v_4$ instead of $f$ and $e$.
 	Hence we can by  interchanges transform $f_1$, and therefore 
 	$f$ too, to a 	required coloring $f'$
	satisfying \eqref{E1}.
	
	If $v_4$ is not an endpoint of $P$, then by Lemma
	\ref{lem:onesided} (b) there is a sequence of 
	interchanges on bicolored paths
	with colors from $\{3,4,5\}$ yielding a coloring $g$, such that
	$c \notin g(v_j) \cup g(v_{j+1})$, for some $j \in\{1,2,3\}$ and some color $c \in \{3,4,5\}$.
	Recolor $v_j v_{j+1}$ with $c$ and denote the obtained coloring by $g_1$.
 If $j\leq 2$, the
	desired result  follows by applying Lemma \ref{lem:reduction} to 
	the coloring $g_1$ instead of $f$. Suppose now that $j=3$.
	We  make an interchange on the maximal $(1,2)$-colored path
	containing $u_1v_1$ and denote the obtained coloring by $g_2$. 
	Then $u_1v_1$ is correct under $g_2$ but $v_3v_4$ is not.
	Furthermore $|M(g_2,1) \cap M(h,1)|\geq |M(f,1) \cap M(h,1)|$ and
$v_3 v_4$ is adjacent to only one edge colored
	$1$ under $g_2$.
Thus we have a situation that is similar to the situation considered in Case A
with the coloring $g_2$ and the edge $v_3v_4$ instead of $f$ and $e$.
 Therefore we can by using interchanges construct a required coloring $f'$ satisfying \eqref{E1}.

\bigskip
\noindent
{\bf Case B.2.}  $P$ is a path with
$P = u_l u_{l-1} \dots u_2 u_1 v_1v_2 \dots v_k$ where $l \geq 4$ and $k\geq 4$, or
$P$ is a cycle of length at least 6 with $P=u_lu_{l-1}...u_1v_1v_2...v_l$, where $l\geq 3$. 
	
\medskip
\noindent
{\bf Case B.2.1}.  At least one of the conditions (i), (ii) of Lemma 2.1 does not hold.

\medskip
\noindent
Then, by Lemma \ref{lem:onesided} (a),
	there is a sequence of interchanges on bicolored paths
	with colors from $\{3,4,5\}$
	which transforms $f$ to a proper $5$-edge coloring $g$, such that $c\notin g(v_j)\cup g(v_{j+1})$
	for some color $c \in \{3,4,5\}$ and some $j\in \{1,2\}$.
	We recolor $v_jv_{j+1}$ with color $c$. If $j=1$, we obtain a situation similar to the situation considered in Case A.
If $j=2$, we obtain a situation similar to the situation considered in Case B.1. Hence in both cases we can transform $f$  to a proper coloring $f'$ satisfying \eqref{E1} by a sequence of interchanges.\bigskip

\medskip
\noindent
{\bf Case B.2.2}. At least 	one of the following conditions does not hold:
	
	\begin{itemize}
	\item[(2.1)] $|f(u_i) \cap f(u_j) \setminus \{1,2\}|= 1$
for $i,j \in \{1,2,3\}$
	satisfying $i \neq j$,
	\item[(2.2)] if $y_1$ and $y_2$
	are the vertices  distinct from $u_1$ and $u_3$
	that are adjacent to $u_2$,
	then colors $3,4,5$ appear at $y_1$ and $y_2$.
	\end{itemize}
	
\medskip
\noindent
	Consider the (1,2)-colored path $v_1u_1u_2u_3u_4$. Arguing similarly as in the Case B.2.1, we can  construct a required coloring $f'$ satisfying \eqref{E1} by a sequence of interchanges.\bigskip

\medskip
\noindent
{\bf Case B.2.3}. The conditions (i), (ii) of Lemma 2.1 hold 
for the vertices $v_1,v_2, v_3$ 
and the conditions (2.1),(2.2) hold for the vertices $u_1, u_2, u_3$.

\medskip
\noindent
Suppose that it is impossible to  transform the coloring $f$ to a required coloring $f'$ 
satisfying \eqref{E1} by using only interchanges.

We will show that  then $f$  can be transformed  
		via interchanges on bicolored paths
		with colors from $\{3,4,5\}$
		to a proper $5$-edge coloring $f_1$ such that

\begin{equation}
f_1(v_1) = f_1(u_1), \quad f_1(v_2) = f_1(u_2), \quad
		f_1(v_3) = f_1(u_3).
\label{E2}
\end{equation}

 Without loss of generality we may assume that
	$$f(v_1) =\{1,2,3,5\}, f(v_2) = \{1,2,4,5\},
	f(v_3) =\{1,2,3,4\}.$$
	
	Suppose first that $f(u_1) \neq f(v_1)$,
	e.g. that $f(u_1) = \{1,2,4,5\}$.
	(The case when $f(u_1) = \{1,2,3,4\}$ is similar.)
	Then, by condition (2.1),  $f(u_1)\not= f(u_2)$.
	This means that either $f(u_2) = \{1,2,3,5\}$ or $f(u_2) = \{1,2,3,4\}$.

	If  $f(u_2) = \{1,2,3,5\}$, then, by condition (2.1), 
	$f(u_3) = \{1,2,3,4\}$. 
	Consider a maximal $(3,4)$-colored path $Q$ with origin at $u_1$.
	We show that $u_2$ is an endpoint of $Q$. Assuming that
	$u_2$ is not an endpoint of $Q$ and  
	interchanging the colors on $Q$ we obtain a coloring $f_2$ such that
	$f_2(u_1) = f_2(u_2) = \{1,2,3,5\}$. Then we may recolor  $u_1u_2$ with color 4 
	and obtain a situation similar to the situation considered in Case A,
	which is impossible because
	in that case we can transform 
	$f$ by interchanges to a required coloring $f'$ satisfying \eqref{E1}.
	Thus, $u_2$ is an endpoint of $Q$.
	Then by interchanging colors
	on $Q$ we obtain the desired coloring $f_1$ satisfying \eqref{E2}.
	
	Now suppose  that  $f(u_2) = \{1,2,3,4\}$.
	Then, by condition (2.1), $f(u_3) = \{1,2,3,5\}$. 
	Let $Q$ be a maximal 	$(4,3)$-colored path
	with origin at $u_1$. Arguing
	similarly as in the preceding paragraph
	we obtain that
	$u_3$ is an endpoint of $Q$.
	By interchanging colors on $Q$ we get 
	a coloring $f^*$.
	Next, we consider a maximal path $Q'$ that 
	is $(3,5)$-colored under $f^*$
	and has origin at $u_2$. 
	Arguing
	similarly as in the preceding paragraph
	we obtain that
	$u_3$ is an endpoint of $Q'$. Now by interchanging 
	colors on $Q'$ we
	get the desired coloring $f_1$ satisfying \eqref{E2}.\medskip

	Suppose now that $f(u_1)=f(v_1)=\{1,2,3,5\}$. 

	If $f(u_2)=f(v_2)=\{1,2,4,5\}$ then, by (2.1), 
	$f(u_3)=\{1,2,3,4\}=f(v_3)$, that is, the coloring $f$ itself can be considered 	as the desired coloring $f_1$ satisfying \eqref{E2}.

	If $f(u_2)\neq f(v_2)$ then the conditions (2.1) and (2.2) 
	imply that 	$f(u_2)=\{1,2,3,4\}$ and
	$f(u_3)=\{1,2,4,5\}$. Let $Q_1$ be a maximal $(3,5)$-colored path
	with origin at $u_2$. Then $u_3$ is an endpoint of $Q_1$. (Otherwise by
	interchanging the colors on $Q_1$ we may obtain a coloring $f_2$ such that
	$f_2(u_3) = f_2(u_2) = \{1,2,4,5\}$. 
	Then we may recolor  $u_2u_3$ with color 3 
	and obtain a situation similar to the 
	situation considered in Case B.1,
	which is impossible because
	in that case we can transform 
	$f$ by interchanges to a required coloring $f'$ satisfying \eqref{E1}.)
	Now by interchanging colors
	on $Q_1$ we obtain the desired coloring $f_1$ satisfying \eqref{E2}.\medskip

	Thus we may,
	without loss of generality, further assume
	that $f$ can by interchanges on bicolored paths
	with colors from $\{3,4,5\}$ be transformed to a coloring $f_1$
	such that
\begin{equation}
f_1(v_1) = f_1(u_1) = \{1,2,3,5\},
	f_1(v_2) = f_1(u_2) = \{1,2,4,5\},
	f_1(v_3) = f_1(u_3) = \{1,2,3,4\}.
	\label{E3}
\end{equation}

\bigskip
Furthermore we can claim that
it is impossible to transform  $f_1$ to a
proper 5-edge coloring $f'$ satisfying \eqref{E1} by using only interchanges, 
because $f_1$ is obtained from $f$  by a sequence of interchanges and, by our assumption,
 it is impossible to transform  $f$ to a
proper 5-edge coloring $f'$ satisfying \eqref{E1} by using only interchanges. 
 	
	Denote by $x_1,x_2$ the vertices  distinct from $v_1$ and $v_3$
	that are adjacent to $v_2$, where $f(v_2x_1) = 5$.
	We now prove that 
	$f_1(x_1)=\{2,3,4,5\}$ and $f_1(x_2) = \{1,3,4,5\}$.
	Suppose that one of these two conditions does not hold.
	Then, by Lemma \ref{lem:x1x2},
	one of conditions (II) and (III) of  Lemma \ref{lem:x1x2}  holds.
 	If the condition (III) holds,  we may transform 
 	$f_1$ by interchanges on paths
	with colors from $\{3,4,5\}$
	to  a proper coloring $g$, such that 
	$c \notin g(v_i) \cup g(v_{i+1})$
	for some color $c \in \{3,4,5\}$ and some $i\in \{1,2\}$.
	Then by recoloring $v_iv_{i+1}$ with $c$
	we  obtain a situation that is similar  to the situation considered in  
	Case A or Case B.1.
	If the condition (II) holds, then, by Lemma \ref{lem:x1x2}, 
	we may transform $f_1$ by interchanges to another proper coloring 
	and again  obtain  a situation which is similar to 
	the situation considered in  Case A.

	Thus if one of conditions $f_1(x_1)=\{2,3,4,5\}$ and 
	$f_1(x_2) = \{1,3,4,5\}$ 	does not hold, then  we obtain  
	situations which are similar to situations
	considered in Case A or Case B.1, which is impossible because
	in those cases we can transform 
	$f_1$ by interchanges to a required coloring $f'$ satisfying \eqref{E1}.
	Therefore $f_1(x_1)=\{2,3,4,5\}$ and $f_1(x_2) = \{1,3,4,5\}$.

	Denote by $y_1$ and $y_2$ the vertices adjacent to $u_2$
	that are  distinct from $u_1$ and $u_3$, and
	where $f_1(u_2y_1) = 5$. Clearly, $x_1\neq y_1$.
	Arguing similarly as in the preceding paragraph
	we may deduce  that
	$f_1(y_1)=\{2,3,4,5\}$ and $f_1(y_2)=\{1,3,4,5\}$.
	
	Hence, without loss of generality we will in the following
	assume that

\begin{equation}
f_1(x_1)=f_1(y_1) = \{2,3,4,5\} 
	\text{ and } f_1(x_2) = f_1(y_2) =\{1,3,4,5\}.
\label{E4}
\end{equation}
	In particular, this implies that  $x_1$ and $y_1$ are not internal vertices of $P$.\bigskip

We will now show that $P$ is not a cycle of length 6 or a path of length 7.

Suppose that $P$ is a cycle of length 6, that is, $P=u_3u_2u_1v_1v_2v_3u_3$. Then the edge $v_3u_3$ is $h$-correct, because otherwise 
by interchanging colors on  $P$ we can obtain a proper coloring $f'$ satisfying (4.1), which is impossible. 
	Using properties (4.3) of $f_1$,
	we can obtain a proper  coloring $f'$ satisfying (4.1)  by first
	recoloring $v_3 u_3$ with color $5$, 
	then interchanging colors on the path $v_3 v_2 v_1 u_1 u_2 u_3$,
	and	thereafter making an interchange on 
	the maximal $(1,5)$-colored path
	$y_1 u_2 u_3 v_3 v_2 x_1$. Since this is impossible, $P$ cannot be a cycle of length 6.

Suppose now that $P$ is a path of length 7. Then $P=u_4u_3u_2u_1v_1v_2v_3v_4$.
Note that $\{v_4, u_4\} \cap \{x_1, y_1\} = \emptyset$,
	because $2 \notin f_1(v_4)$, $2 \notin f_1(u_4)$
	and $2 \in f_1(x_1)$, $2\in f_1(y_1)$.
	Make an interchange on $P$ and denote the obtained coloring by $f_2$.
	Then $u_3 u_2 y_1$ and $v_3 v_2 x_1$ are two disjoint maximal
	$(1,5)$-colored paths under $f_2$ in $G$. 
	By interchanging colors on these
	paths and then recoloring $u_4u_3$ and $v_4v_3$ with color $1$
	we obtain  a proper coloring $f'$ satisfying (4.1). 
	Since this is impossible, $P$ cannot be a path of length 7.\bigskip
	
	Thus $P$ has the length at least 8.
 	We will show now that  the edges $v_3v_4$ and $u_3u_4$ 
 	are $h$-correct under $f_1$.

	 Since $P$ has length at least $8$, at most one of $u_4$ and
	$v_4$ are endpoints of $P$. Suppose first that $u_4$ is
	an endpoint of $P$. If $u_3u_4$ is not correct, then we may proceed
	exactly as in Case B.1 to obtain a coloring $f'$
	satisfying \eqref{E1}, which contradicts our assumption
	that this is impossible. 
	Hence, we may assume that $u_3u_4$ is correct.
	
	If $v_3v_4$ is not correct, then by applying Lemma
	\ref{lem:onesided} (b) to the path $u_1 v_1 v_1 v_2 v_3 v_4 v_5$
	we get that there is 
	a sequence of interchanges on bicolored paths
	with colors from $\{3,4,5\}$ yielding a coloring $g$, such that
	$c \notin g(v_j) \cup g(v_{j+1})$ 
	for some $j \in\{1,2,3\}$ and some color $c\in \{3,4,5\}$. 
	By recoloring $v_j v_{j+1}$ with $c$,
	we obtain  a situation which is similar to one of 
	the situation considered  in 
	Case B.1 and Case A. This is impossible because in those cases
	we can  obtain, by interchanges,  
	a required coloring $f'$ satisfying \eqref{E1}.
	
	Suppose now that none of $u_4$ and $v_4$ is an endpoint of $P$.
	If 
	$u_3u_4$ (or $v_3v_4$) is not correct, then by applying Lemma
	\ref{lem:onesided} (b) to the path $v_1 u_1 u_2 u_3 u_4 u_5$
	(or $u_1 v_1 v_2 v_3 v_4 v_5$, respectively) and proceeding as in the preceding
	paragraph
	%
	%we get that there is 
	%a sequence of interchanges on bicolored paths
	%with colors from $\{3,4,5\}$ yielding a coloring $g$, such that
	%$c \notin g(u_j) \cup g(u_{j+1})$ for some $j \in\{1,2,3\}$ 
	%and some color $c\in \{3,4,5\}$. 
	%By recoloring $u_ju_{j+1}$ with $c$,
	we obtain  a situation which is similar to one of the situations considered  in 
	Case B.1 and Case A. This is impossible because in those cases
	we can  obtain, by interchanges,  a required coloring 
	$f'$ satisfying \eqref{E1}. 
	
	%Thus we may assume that $u_3u_4$ is correct under $f_1$. 
	%Arguing similarly for the path $u_1v_1v_2v_3v_4$, we can 
	%deduce that  the edge $v_3v_4$ is also
	%$h$-correct. 
	Thus we have that

\begin{equation}
\text {the edges}  \  u_3u_4 \ \text {and} \ v_3v_4  \ \text  {are correct under} f_1
\label{E5}
\end{equation}

	We now prove the following claim that holds for the coloring $f_1$ in Case B.2.3
	when $P$ has the length at least $8$.
	
\begin{claim}
\label{cl:paths}

\begin{itemize}

	\item[(a)]	If $Q$ is a maximal $(3,5)$-colored path under $f_1$
	with origin at $u_3$ (at $v_3$), then $u_2$ (respectively, $v_2$)
	is an endpoint of $Q$.
	
	\item[(b)] If $Q'$ is a maximal $(4,5)$-colored path under $f_1$
	with origin at $u_3$ (at $v_3$), then 
	
	\begin{itemize}
		
\item[(i)] $u_1$ (respectively, $v_1$) is an endpoint of $Q'$,

	\item[(ii)] $Q'$ passes through $u_2$ (through  $v_2$),
	and
	
	\item[(iii)] $Q'$ does not pass through $v_2$ (through $u_2$).
	
	\end{itemize}
	\end{itemize}	
\end{claim}
\begin{proof}
	We first prove (a).
	Let $Q$ be a maximal $(3,5)$-colored path 
	with origin at $u_3$. Suppose that  $u_2$ is not an endpoint of $Q$.
	Then, by making an
	interchange on $Q$, we obtain a proper coloring $f^*$
	satisfying  $f^*(u_3) = f^*(u_2)$. Thus  we
 	obtain a situation similar to the situation considered in Case B.2.2,
	which is impossible because
	in that case we can transform 
	$f$ by interchanges to a required coloring $f'$ satisfying \eqref{E1}.
	Hence $u_2$ is  an endpoint of $Q$.
	
	We now prove (b).
	Let $Q'$ be a maximal path that is $(4,5)$-colored
	under $f_1$ and has origin at $u_3$.
	By interchanging colors on this path
	we obtain a coloring $f_2$.

	Then  $u_1$ is  an endpoint of $Q'$, because 
	otherwise the coloring $f_2$ will satisfy  $f_2(u_3)=f_2(u_1)$,
	which is impossible by the same reason as in  
	the proof of the part  (a) of this claim.

	Suppose that $u_2\notin V(Q')$. Then	 by applying Lemma \ref{lem:x1x2} 
	to the path $v_1u_1u_2u_3 u_4$, we deduce that
	one of conditions (II) and (III) of  Lemma \ref{lem:x1x2}  holds.
 If the condition (III) holds,  we may transform $f_2$ by interchanges 
on paths
	with colors from $\{3,4,5\}$
	to  a proper coloring $g$, such that $c \notin g(u_i) \cup g(u_{i+1})$
	 for some color $c \in \{3,4,5\}$, and some $i\in \{1,2\}$.
	Then by recoloring $u_iu_{i+1}$ with $c$
	we  obtain a situation that is similar  to the situation considered in  Case A or Case B.1.
If the condition (II) holds, then, by Lemma \ref{lem:x1x2}, we may transform $f_2$ by interchanges to another proper coloring 
and again  obtain  a situation which is similar to the situation considered in  Case A.

Thus if $u_2\notin V(Q')$, then  we obtain  a situation which is similar to one of the situations
considered in Case A and Case B.1, which is impossible because
in those cases we can transform 
$f$ by interchanges to a required coloring $f'$ satisfying \eqref{E1}.
Therefore $u_2\in V(Q')$.

	Finally,  $v_2 \notin V(Q')$, because if   $v_2 \in V(Q')$, then
	$v_1v_2x_1$ is a maximal 
	$(1,4)$-colored path under $f_2$, and
	by making an interchange on this path
	we obtain a coloring where only
	one edge adjacent to $u_1v_1$ is colored $1$,
	thus reducing this situation to the situation considered in Case A.
	But this is impossible because
in that case we can transform 
$f$ by interchanges to a required coloring $f'$ satisfying \eqref{E1}.
\end{proof}	
\bigskip

We continue the proof of the theorem. 
We have that the proper 5-edge coloring $f_1$ satisfies the conditions \eqref{E3}, \eqref{E4}, \eqref{E5} and
$P$ is either a cycle of length at least 8, or a path  with
$P = u_l u_{l-1} \dots u_2 u_1 v_1v_2 \dots v_k$ where $l \geq 4$, $k\geq 4$ and $k+l\geq 9$.	
 We assumed that   it is impossible, using only interchanges, to transform  $f$ (and therefore $f_1$ too)  to a
proper 5-edge coloring $f'$ satisfying \eqref{E1}. 
Now we will show that  for all possible $P$, despite our assumption, we can transform $f_1$  by interchanges 
either directly to a proper coloring $f'$ satisfying (4.1), or to another proper coloring $\varphi$ where we have a situation which is similar to one of situations considered in 
Case A and Case B.1. This will lead to a contradiction, because in those cases 
we can transform 
$\varphi$, and therefore $f_1$ too, by a sequence of interchanges to a proper coloring $f'$ satisfying \eqref{E1}.

\bigskip

\noindent
{\bf Case B.2.3.1.} $P$ is a path and only one of the vertices $u_4$ and $v_4$ is an endpoint of $P$.
	
\medskip	
\noindent
 	Without loss of generality we assume that $u_4$  is an endpoint of $P$.
	If $f_1(v_4) = \{1,2,3,4\}$,
	then we recolor $v_3v_4$ with color $5$, and thereafter make
	an interchange on $u_4 u_3 u_2 u_1 v_1 v_2 v_3$ and denote the obtained
	coloring by $f_2$.
	Then $u_3 u_2 y_1$ is a maximal $(1,5)$-colored path  under $f_2$. 
	By interchanging colors on this path
	and then recoloring $u_4u_3$ with color $1$ we obtain a coloring where
	$u_4u_3$ and $u_1v_1$ are correct 
	but not $v_3v_4$. However, $v_3v_4$ is incident to only
	one edge colored $1$. We have thus obtained a situation which is   
	is similar to one of the	situations considered in Case A.
		
	Suppose now that $f_1(v_4) = \{1,2,3,5\}$.
	Then we consider a maximal path $Q$ that is $(4,5)$-colored 
	under $f_1$ and has origin at $v_3$.
	It follows from Claim \ref{cl:paths} that
	 $v_1$ is an endpoint of $Q$
	and 
	$u_2 \notin V(Q)$.
	By interchanging colors on $Q$ we obtain a proper coloring
	that we denote by $f_2$.
	Then we recolor $v_3v_4$ with $4$, make an interchange
	on the maximal $(1,2)$-colored path $u_4 u_3 u_2 u_1 v_1 v_2 v_3$
	and denote the obtained coloring by $f_3$.
	Then $u_3 u_2 y_1$ is a maximal
	$(1,5)$-colored path under $f_3$ in $G$. 
	By interchanging colors on this path
	and then recoloring $u_3u_4$ with $1$,
	we obtain a coloring where
	$u_4u_3$ and $u_1v_1$ are correct 
	but not $v_3v_4$.
	As before, $v_3v_4$ is incident to only
	one edge colored $1$. Hence, we obtain a situation which is similar 
	to one of the situations considered in  Case A.
	
	Suppose now that $f_1(v_4) = \{1,2,4,5\}$.
	Then we consider a maximal $(3,5)$-colored path 
	$Q$ with origin at $v_3$. 
	By Claim \ref{cl:paths},
	$v_2$ is an endpoint of $Q$.
	By interchanging colors on $Q$, we obtain a coloring that we denote by
	$f_2$. We now recolor $v_3v_4$ with $3$ and
	then interchange colors on
	the maximal $(1,2)$-colored
	path $u_4u_3u_2u_1v_1v_2v_3$ and denote the obtained coloring by $f_3$.
	Then
	$u_3 u_2 y_1$ is a maximal
	$(1,5)$-colored path under $f_3$. By interchanging colors on this
	path and then recoloring $u_3u_4$ with color $1$
	we obtain a situation similar to the situation considered in Case A.
\bigskip

\noindent
{\bf Case B.2.3.2.} $P$ is a cycle of length at least 8 or a path where none
of the vertices $u_4$ and $v_4$ is an endpoint
	of $P$.
	
	\noindent Since $\{1,2\}\subset f_1(v_4)$ and $\{1,2\}\subset f_1(u_4)$, the set $f_1(v_4)$
as well as the set  $f_1(u_4)$ is one of the sets $\{1,2,3,4\}$, $\{1,2,3,5\}$ and $\{1,2,4,5\}$.
Furthermore, by symmetry in (4.3), we have to consider the following six cases only: 
\begin{itemize}

 \item $f_1(v_4) = f_1(u_4) = \{1,2,3,4\}$, 

 \item $f_1(v_4) = f_1(u_4) = \{1,2,3,5\}$,

 \item  $f_1(v_4) = f_1(u_4) = \{1,2,4,5\}$, 

 \item $f_1(u_4)  = \{1,2,3,4\}$ and  $f_1(v_4)  = \{1,2,3,5\}$,

 \item $ f_1(u_4) = \{1,2,3,4\}$ and $f_1(v_4)  = \{1,2,4,5\}$, 

 \item $f_1(u_4)  = \{1,2,3,5\}$ and  $f_1(v_4)  = \{1,2,4,5\}$. \bigskip

\end{itemize}

	$f_1(v_4) = f_1(u_4) = \{1,2,3,4\}$: 	
	Recolor $v_3v_4$ and $u_3u_4$ with $5$
	and then make an interchange on the maximal $(1,2)$-colored path
	containing $u_1v_1$. Then $u_4u_3u_2 y_1$ and $v_4 v_3 v_2 x_1$
	are maximal disjoint $(1,5)$-colored paths. By interchanging
	colors on these paths we obtain a proper  
	coloring $f'$ satisfying 	(4.1).\bigskip
	
	$f_1(u_4)= f_1(v_4) = \{1,2,3,5\}$:
	Let $Q_1$ and $Q_2$ be maximal $(4,5)$-colored paths with origin
	at $u_3$ and $v_3$, respectively. 
	Then by Claim \ref{cl:paths}, 
	$u_1$ is an endpoint of $Q_1$, $v_1$ is an endpoint of $Q_2$,
	$u_2 \notin V(Q_2)$ and $v_2 \notin V(Q_1)$.
	Note also that this implies that $Q_1$ and
	$Q_2$ are disjoint.
	After interchanging colors on
	$Q_1$ and $Q_2$ we may recolor $u_3u_4$ and $v_3v_4$ with $4$.
	Next, interchange colors on the maximal $(1,2)$-colored path
	$u_3u_2u_1v_1v_2v_3$. By Claim 1, $u_2\in V(Q_1)$ and $v_2\in V(Q_2)$,
	which implies that both $u_4u_3u_2y_1$ and $v_4v_3v_2x_1$
	are maximal $(1,4)$-colored paths, and by interchanging colors
	on these paths we obtain a proper coloring $f'$ satisfying (4.1).\bigskip
	
	$f_1(u_4)= f_1(v_4) = \{1,2,4,5\}$:
	We proceed similarly as in the preceding paragraph, but instead
	consider $(3,5)$-colored paths $Q_1$ and $Q_2$
	with origin at $u_3$
	and $v_3$, respectively. 
	By Claim \ref{cl:paths},
	$Q_1$
	and $Q_2$ have endpoints $u_2$ and $v_2$, respectively.
	Note also that this implies that $Q_1$ and
	$Q_2$ are disjoint.
	Interchange colors on $Q_1$ and $Q_2$, 
	and then recolor $u_3u_4$
	and $v_3v_4$ with $3$. Next, we make an interchange 
	on the maximal $(1,2)$-colored path containing $u_1v_1$
	and denote the obtained coloring by $f_2$. 
	The paths $u_4 u_3 u_2 y_1$ and
	$v_4 v_3 v_2 x_1$ are
	maximal disjoint $(1,3)$-colored paths under $f_2$. 
	Interchange colors on these paths
	to obtain a proper coloring $f'$ satisfying (4.1).
	\bigskip

$f_1(u_4) = \{1,2,3,4\}$ and $f_1(v_4) = \{1,2,3,5\}$:
	Then we consider a maximal $(4,5)$-colored path $Q$
	under $f_1$ with origin at $v_3$.
		By Claim \ref{cl:paths},
	$v_1$ is an endpoint of $Q$,
	$v_2 \in V(Q)$ and $u_2 \notin V(Q)$.
	First  by interchanging colors on $Q$,
	and thereafter recoloring $u_3u_4$ with $5$ and $v_3v_4$ with $4$
	we obtain a coloring $f_2$. Next,
	we interchange colors on the maximal $(1,2)$-colored
	path $u_3u_2u_1v_1v_2v_3$
	and denote the obtained coloring by $f_3$.
	Then $u_4 u_3 u_2 y_1$ is a maximal $(5,1)$-colored path
	under $f_3$
	and by interchanging colors on this path
	we obtain a proper coloring 
	$f_4$ under which
	$u_4u_3$ and $u_1v_1$ are correct, but not $v_3v_4$.
	However,  $v_3v_4$ is adjacent to only one edge colored $1$,
	so we have obtained a situation which is similar
	to the one considered in Case A.	
	%the coloring $f'$ satisfies the condition (4.1), because 
%$M(f',1)\cap M(h,1)=(M(f,1)\cap M(h,1)\setminus \{v_3v_4\})
%\cup \{v_1u_1, u_3u_4\}$.
\bigskip

	$f_1(u_4) = \{1,2,3,4\}$ and $f_1(v_4) = \{1,2,4,5\}$:
	Then we consider a maximal path $Q$ that is $(3,5)$-colored under
	$f_1$ and has origin at $v_3$. 
	By Claim \ref{cl:paths},
	$v_2$ is an endpoint of $Q$.
	We now  interchange colors
	on $Q$, then recolor
	$u_3u_4$ and $v_3v_4$ with $5$ and $3$, respectively,
	and denote the obtained coloring by $f_2$. 
	Thereafter make  an interchange on the maximal
	$(1,2)$-colored path $u_3u_2u_1v_1v_2v_3$.
	Then $u_4 u_3 u_2 y_1$ is a maximal $(5,1)$-colored path
	under the new coloring $f_3$. By interchanging colors on this path
	we obtain a proper coloring under which
	$u_4u_3$ and $u_1v_1$ are correct, but not $v_3v_4$.
		However,  $v_3v_4$ is adjacent to only one edge colored $1$,
	so we have obtained a situation which is similar
	to the one considered in Case A.	

\bigskip

	$f_1(u_4)=  \{1,2,3,5\}$ and $f_1(v_4) = \{1,2,4,5\}$:
	Let $Q_1$ be a maximal $(4,5)$-colored path with origin
	at $u_3$. Then, by Claim \ref{cl:paths},
	$u_1$ is an endpoint
	of this path, and $u_2  \in V(Q_1)$ and $v_2 \notin V(Q_1)$.
	Interchange colors on $Q_1$ and denote the obtained coloring
	by $f_2$.
	Next, we consider a maximal $(3,5)$-colored path $Q_2$ 
	with origin at $v_3$. Similarly as in proof of Claim 1, it can be shown
	that $v_2$ is an endpoint of $Q_2$.
	Interchange colors on $Q_2$
	and denote the obtained coloring by $f_3$.
	We recolor $v_3v_4$ and $u_3 u_4$ with colors $3$ and $4$,
	respectively, and thereafter interchange colors on the maximal
	$(1,2)$-colored path $u_3 u_2 u_1v_1 v_2 v_3$. Denote this new coloring by $f_4$.
	Then $v_4 v_3 v_2 x_1$ is a maximal $(1,3)$-colored path under $f_4$.
By  interchanging the colors on this path
	  we  obtain a situation which is similar to the  situation considered in
	 Case A. \bigskip

We showed that for all possible $P$ we can transform $f_1$  by interchanges 
either directly to a proper coloring $f'$ satisfying (4.1), or to another proper coloring $\varphi$ where we have  a situation which is similar to one of situations considered in 
Case A and Case B.1.  In both cases we can reach  a proper  coloring $f'$ satisfying \eqref{E1} by a sequence of interchanges.
Since $f_1$ was obtained from $f$  also by a sequence of interchanges, we can obtain $f'$ from $f$ by interchanges too. This lead to a contradiction with our assumption 
 that  it is impossible, using only interchanges, to transform  $f$  to a
proper 5-edge coloring $f'$ satisfying \eqref{E1}.
\end{proof}	
\bigskip

\begin{proof}
[{\bf Proof of Theorem	\ref{th:main}}]
	Let $G$ be a graph with $\chi'(G) = \Delta(G)=4$. Furthermore,
	let $h$ be a proper $4$-edge coloring of $G$ and $f$ be a proper 
	$t$-edge coloring of $G$, where $t\geq 5$. We will show that  
	$f$ can be transformed to $h$ by
	a sequence of interchanges.
	
	If $t > 5$,
	then by  Vizing's algorithm $f$ can
	be transformed into a proper $5$-edge coloring
	using only interchanges (see e.g. \cite{West});
	the original problem now is reduced
	to the problem of transforming the obtained proper 5-edge coloring to $h$.
	Thus we may assume that $t = 5$.

Let $m$ denote the minimum degree of $G$, $1\leq m\leq 4$. We define graphs 
$G_0,G_1,...,G_{4-m}$ and for each $G_i$ two  proper colorings $f_i$ and $h_i$ in the following way:

 1) $G_0=G$ and $f_0=f$, $h_0=h$.

2) Suppose that $G_i, f_i$ and $h_i$ are already defined 
for some $i\geq 0$. If $i=4-m$, then stop. Otherwise take
two disjoint copies  $G_i'$
and $G_i''$ of $G_i$,  with $V(G_i')=\{x':x\in V(G_i)\}$ and $V(G_i'')=\{x'':x\in
V(G_i)\}$. Then we define  $G_{i+1}$ as the graph
obtained from $G_i'$ and $G_i''$ by joining $x'$ and $x''$ with
an edge, for each vertex $x\in V(G_i)$ with degree less than 4.
The coloring $f_i$ of $G_i$ induces a proper $5$-edge coloring $f_{i+1}$ of $G_{i+1}$
in the following way: we color the copies $G_i'$ and $G_i''$ in the same
way as $G_i$, and then color the  edge joining $x'$ to
$x''$ with a color  from $\{1,2,3,4\}$ which is not used to
color an edge incident with $x$ in $G_i$. Similarly, the coloring $h_i$ of $G_i$
induces a proper $4$-edge coloring $h_{i+1}$ on $G_{i+1}$. 

Clearly, the graph $G_{4-m}$ is a 4-regular graph. Furthermore, $h_i$ is a proper 4-edge coloring and  
$f_{i}$ is a proper 5-edge coloring of $G_i$, for $i=0,1,\dots,4-m$.
Then, by Theorem 4.1,   the proper 5-edge coloring $f_{4-m}$ of the graph $G_{4-m}$ can be transformed into the proper 4-edge coloring $h_{4-m}$ of $G_{4-m}$ by a
sequence of interchanges. 
It is clear that these interchanges
 also define a sequence  of  proper colorings of $G$,
beginning with $f$ and ending with $h$, such that each  intermediate coloring   in this sequence   either is the same as  the previous coloring or obtained from the previous coloring   by either one interchange or a sequence of
interchanges.

This completes the proof of Theorem \ref{th:main}.
\end{proof}

\section{Proof of Theorem \ref{th:main3}}

In this section we prove Theorem \ref{th:main3}.

\begin{proof}
[{\bf Proof of Theorem \ref{th:main3}}]
Every graph $G$ with $\Delta(G)\geq 5$ where $G_{\geq 5}$ is acyclic  is a Class 1 graph because the graph $G_{\Delta}$ is  acyclic.

Let $\cal A$ denote the set of graphs $G$ with $\Delta(G)\geq 5$ where $G_{\geq 5}$ is acyclic and all Class 1 graphs with maximum degree 4.
  We will show that for every $G\in \cal A$ all proper
$(\Delta(G)+1)$-edge colorings are Kempe equivalent.
The proof is by induction on the maximum degree $\Delta(G)$. 

 By Corollary 1.4, the proposition is true  for all graphs in $\cal A$  with maximum degree 4. Suppose that the proposition  is true for all graphs in $\cal A$ with maximum degree  less than $k$, 
 $k\geq 5$, and let $G\in \cal A$  be a graph with $\Delta(G)=k$. Then 
  $\chi'(G)=k$.
Consider  a proper $k$-edge coloring $h$ of $G$. As in the proof of Theorem \ref{th:connect}
 we can assume that $M(h,k)$ is a maximal matching of $G$. 

In order to show that any two proper $(k+1)$-edge colorings of $G$ are Kempe equivalent
 it is sufficient to show that any proper $(k+1)$-edge coloring  of $G$ can be transformed to $h$
 by a sequence of interchanges. 
Let $\varphi$ be an arbitrary proper $(k+1)$-edge coloring of $G$. By Proposition \ref{th:main2}, $\varphi$
can be   transformed  by a sequence of interchanges to a proper $k$-edge coloring  $\varphi_1$.
Then the coloring $\varphi_1$ can be transformed  to a proper
$(k+1)$-edge coloring  by sequentially recoloring the edges in the matching  $M(h,k)$  with color $k+1$. Clearly,  the obtained 
proper $(k+1)$-edge coloring $f$ satisfies the condition $M(h,k)=M(f,k+1)$. Furthermore,
$f$ is obtained from $\varphi$ by a sequence of interchanges.

Consider the graph $G_1=G-M(h,k)$.  
	The proper $k$-edge coloring $h$ of $G$ induces a proper 
	$(k-1)$-edge  coloring $h'$ of $G_1$ and the proper 
	$(k+1)$-edge 	coloring $f$ of $G$ induces a proper 
	$k$-edge  coloring $f'$ of $G_1$. Clearly, $\chi'(G_1)=k-1$.
	Therefore, $h'$ is a proper $\chi'(G_1)$-edge coloring and $f'$ is a proper
	$(\chi'(G_1)+1)$-edge coloring of $G_1$.

Moreover, $\Delta(G_1)=k-1$, because $G$ is a Class 1 graph and $k=\chi'(G)$. 
Furthermore, $G_1\in \cal A$. Then, by the induction hypothesis, $f'$ and $h'$ are Kempe equivalent. This implies that  the coloring $f$  of $G$ can be transformed to the coloring $h$ by a sequence of interchanges.

Thus for any graph $G\in \cal A$ with $\Delta(G)=k$ all proper
$(k+1)$-edge colorings are Kempe equivalent.
The proposition follows by the principle of induction.
\end{proof}

\end{document}